\documentclass[preprint,number, 3p]{elsarticle}

\usepackage[utf8]{inputenc}
\usepackage[english]{babel}
\usepackage{etoolbox}
\usepackage{amsmath}
\usepackage{amsfonts}
\usepackage{amssymb}
\usepackage{amsthm}
\usepackage{enumitem}
\usepackage{relsize}
\usepackage{thmtools}
\usepackage{dsfont}
\usepackage{bm}
\usepackage{hyperref}

\makeatletter
\def\ps@pprintTitle{%
 \let\@oddhead\@empty
 \let\@evenhead\@empty
 \def\@oddfoot{\centerline{\thepage}}%
 \let\@evenfoot\@oddfoot}
\makeatother

\allowdisplaybreaks

\newcommand{\eps}{\varepsilon}
\renewcommand{\Re}{\mathrm{Re}\,}
\renewcommand{\Im}{\mathrm{Im}\,}
\newcommand{\ma}{\mathrm{m}_a}
\newcommand{\conv}[1]{
	\mskip-1.8\thinmuskip\smash{\begin{array}[t]{c} \mathlarger{\ast} \\[-10pt] \mathsmaller{#1} \end{array}}\mskip-1.2\thinmuskip
	}
\newcommand{\ccdot}{\kern-.12em\cdot\kern-.12em}
\DeclareMathOperator*{\esssup}{ess\,sup}

\makeatletter
	\@ifdefinable\@latex@chi{\let\@latex@chi\chi}
	\renewcommand*\chi{{\@latex@chi\smash[t]{\mathstrut}}} 
\makeatother

\newtheoremstyle{slplain}
  {0.4cm}
  {0.4cm}
  {\upshape}
  {}
  {\bfseries}
  {.}
  { }
  {}
  
\newtheoremstyle{itplain}
    {0.4cm}
    {0.4cm}
    {\itshape}
    {}
    {\bfseries}
    {.}
    { }
    {}
  
\declaretheorem[style=slplain,numberwithin=section]{definition}

\declaretheorem[style=slplain,sibling=definition]{remark}

\declaretheorem[style=itplain,sibling=definition]{theorem}

\declaretheorem[style=itplain,sibling=definition]{proposition}
\declaretheorem[style=itplain,sibling=definition]{lemma}
\declaretheorem[style=itplain,sibling=definition]{corollary}

\patchcmd{\MaketitleBox}{\footnotesize\itshape\elsaddress\par\vskip36pt}{\footnotesize\itshape\elsaddress\par\parbox[b][36pt]{\linewidth}{\vfill\hfill\textnormal{\today}\hfill\null\vfill}}{}{}%
\patchcmd{\pprintMaketitle}{\footnotesize\itshape\elsaddress\par\vskip36pt}{\footnotesize\itshape\elsaddress\par\parbox[b][36pt]{\linewidth}{\vfill\hfill\textnormal{\today}\hfill\null\vfill}}{}{}%

\begin{document}

\begin{frontmatter}

\renewcommand*{\thefootnote}{\fnsymbol{footnote}}
\title{\mbox{ }\\[0.5cm] On the product formula and convolution\linebreak associated with the index Whittaker transform\footnote[2]{The final version of this manuscript is in press in J.\ Math.\ Anal.\ Appl.\, DOI: \href{http://dx.doi.org/10.1016/j.jmaa.2019.03.009}{10.1016/j.jmaa.2019.03.009}\vspace{2.5pt}}\vspace{5pt}}

\author[myaddress1]{Rúben Sousa\corref{mycorrespondingauthor}}
\cortext[mycorrespondingauthor]{Corresponding author}
\ead{rubensousa@fc.up.pt}

\author[myaddress2]{Manuel Guerra}
\ead{mguerra@iseg.ulisboa.pt}

\author[myaddress1]{Semyon Yakubovich}
\ead{syakubov@fc.up.pt}

\address[myaddress1]{CMUP, Departamento de Matemática, Faculdade de Ciências, Universidade do Porto,\linebreak Rua do Campo Alegre 687, 4169-007 Porto, Portugal\vspace{2pt}}

\address[myaddress2]{CEMAPRE and ISEG (School of Economics and Management), Universidade de Lisboa,\linebreak Rua do Quelhas, 1200-781 Lisbon, Portugal}

\begin{abstract}
We deduce a product formula for the Whittaker function $W_{\kappa,\mu}$ whose kernel does not depend on the second parameter. Making use of this formula, we define the positivity-preserving convolution operator associated with the index Whittaker transform, which is seen to be a direct generalization of the Kontorovich-Lebedev convolution. The mapping properties of this convolution operator are investigated; in particular, a Banach algebra property is established and then applied to yield an analogue of the Wiener-Lévy theorem for the index Whittaker transform. We show how our results can be used to prove the existence of a unique solution for a class of convolution-type integral equations.
\end{abstract}


\end{frontmatter}

\vspace{-6pt}

\section{Introduction}

\vspace{-1.5pt}

Let $\{\chi_\lambda\}_{\lambda \in \Lambda}$ be a family of continuous functions on an interval $I \subset \mathbb{R}$, where $\Lambda$ is some indexing set. A functional identity of the form \vspace{-1.5pt}
\[
\chi_\lambda(x)\, \chi_\lambda(y) = \int_I \chi_\lambda(\xi)\, \mathcal{K}(x,y,\xi)\, d\xi,
\]
where the kernel $\mathcal{K}(x,y,\xi)$ does not depend on $\lambda$, is called a \emph{product formula} or \emph{multiplication formula} for $\{\chi_\lambda\}_{\lambda \in \Lambda}$.

Such product formulas are a very useful tool in the theory of special functions. For instance, the existence of product formulas with positive kernels for certain systems of, say, orthogonal polynomials allows us to simplify various problems concerning the positivity of special functions \cite{gasper1975}. Moreover, and most importantly, product formulas are the key ingredient for introducing the so-called \emph{generalized translation} and \emph{generalized convolution} operators, whose theory was initiated by J.\ Delsarte \cite{delsarte1938} and B.\ Levitan \cite{levitan1964}, and which are deeply related with eigenfunction expansions with respect to systems of (orthogonal) special functions, much like the ordinary translation and convolution is closely connected with the Fourier transform. For discussion and examples see e.g.\ \cite{berezansky1998,connett1992}.

The \textit{index Whittaker transform} is the integral transform (of index type) defined by
\begin{equation} \label{eq:intro_Whit_transf}
(\mathrm{W}_{\!\alpha} g)(\tau) := \int_0^\infty g(x) W_{\alpha, i\tau}(x)\, x^{-2} dx, \qquad \tau \geq 0
\end{equation}
where $i$ is the imaginary unit, $\alpha < {1 \over 2}$ is a parameter and $W_{\alpha, \nu}(x)$ is the Whittaker function (cf.\ Section  \ref{chap:preliminaries}). This transformation first appeared in \cite{wimp1964} as a particular case of an integral transform having the Meijer-G function in the kernel. Its $L_p$ theory was studied in \cite{srivastava1998}. In its general form, the index Whittaker is connected with the Asian option pricing problem in mathematical finance \cite{linetsky2004}. Furthermore, it includes as a particular case the Kontorovich-Lebedev transform, which is one of the most well-known index transforms \cite{yakubovichluchko1994,yakubovich1996} and has a wide range of applications in physics (see e.g.\ \cite{bernard2008,gasaneo2001}).

In \cite{sousayakubovich2018} it is observed that the index Whittaker transform is a generalized Fourier transformation, as it constitutes the eigenfunction expansion of a self-adjoint Sturm-Liouville operator. For the particular case of the Kontorovich-Lebedev transform, whose kernel is the modified Bessel function of the second kind $K_\nu(x)$ (also known as the Macdonald function), the product formula for $K_\nu(x)$ is well-known --- it is given by the Macdonald formula, which can be found in standard texts on special functions such as \cite{erdelyiII1953} --- and has been used to introduce the Kontorovich-Lebedev convolution, which was introduced by Kakichev in \cite{kakichev1967} and has been an object of much interest \cite{hoang2017,hong2016,prasadmandal2016,yakubovichluchko1994,yakubovich2004}. Given that the properties of the index Whittaker transform in the general case are similar to those of the Kontorovich-Lebedev transform, one would expect that the index Whittaker transform is also associated with a convolution with analogous properties. However, to the best of our knowledge, neither an explicit product formula for $W_{\alpha,\nu}(x)$ with kernel not depending on the transform variable $\nu$ (the key ingredient for such a convolution) is known in the literature, nor the existence of such a formula has been deduced through techniques such as those described in \cite{connett1992}.

Our main result is to establish a product formula of the form
\[
W_{\alpha,\nu}(x) W_{\alpha,\nu}(y) = \int_0^\infty W_{\alpha,\nu}(\xi)\, \mathcal{K}_\alpha(x,y,\xi)\, d\xi \qquad\;\; (\alpha, \, \nu \in \mathbb{C})
\]
whose kernel $\mathcal{K}_\alpha(x,y,\xi)$ is given in closed form in terms of the parabolic cylinder function. For fixed $x$ and $y$, the kernel $\mathcal{K}_\alpha(x,y,\cdot)$ has full support in $(0,\infty)$, making it clear that the partial differential equation approach that has been used to prove many product formulas for Sturm-Liouville eigenfunctions would not applicable here (cf.\ \cite{gorlich1994}). In the case $\alpha < {1 \over 2}$ the kernel turns out to be strictly positive, and this allows us to construct a positivity-preserving generalized convolution operator $\conv{a}$ which is related with the index Whittaker transform via the identity $\bigl(\varPsi_a(f \conv{a} g) \bigr)(\tau) = (\varPsi_a f)(\tau)\, (\varPsi_a g)(\tau)$, where $\varPsi_a$ is a reparameterized version of the transform \eqref{eq:intro_Whit_transf} (which will be defined in Section \ref{chap:Whit_convol}). Moreover, we show that the connection between the product formula, the index transform $\varPsi_a$ and the convolution $\conv{a}$ yields various mapping properties for the convolution operator.

The Kontorovich-Lebedev convolution has been applied to the study of existence and uniqueness of solutions for integral equations of generalized convolution type (\cite{yakubovichluchko1994}; see also \cite{hong2016}). Here we extend the method to the general index Whittaker transform: we prove that an analogue of the Wiener-Lévy theorem holds for the index Whittaker transform, and we discuss its application to a class of integral equations of the second-kind; in addition, an example is provided where this method yields an explicit expression for the solution of an integral equation with the Whittaker (or the confluent hypergeometric) function in the kernel.

The structure of the paper is as follows. Section \ref{chap:preliminaries} sets notation and collects some basic facts about special functions which will be of use in the sequel. The product formula for the Whittaker function $W_{\alpha,\nu}(x)$ is stated and proved in Section \ref{chap:Whit_prodform}. Section \ref{chap:Whit_convol} is dedicated to the index Whittaker convolution operator: in Subsection \ref{sec:Whit_convol_gentransl} we start by establishing the relevant properties of the generalized translation operator associated with the product formula, which is then used in Subsection \ref{sec:Whit_convol_Lpconv} to define the generalized convolution and to derive its main mapping properties; then Subsection \ref{sec:Whit_convol_banachalg} focuses on the Banach algebra property of the convolution in a family of weighted $L_1$ spaces, from which the analogue of the Wiener-Lévy theorem for the index Whittaker transform is deduced. Finally, Section \ref{chap:Whit_integraleq} treats the application of our results to integral equations.

\section{Preliminaries} \label{chap:preliminaries}

In this paper, the space of continuous functions on the half line $(0,\infty)$ will be denoted by $\mathrm{C}(0,\infty)$, and the notations $\mathrm{C}_\mathrm{b}(0,\infty)$, $\mathrm{C}_\mathrm{c}(0,\infty)$ will stand for its subspaces consisting, respectively, of bounded functions and of compactly supported functions. As usual, $L_p(E; w(x)\, dx)$ denotes the weighted $L_p$-space with norm
\[
\|f\|_{\scalebox{0.65}{$L_p(E; w(x)\, dx)\!$}} = \biggl( \int_E |f(x)|^p w(x)dx \biggr)^{1/p} \quad (1 \leq p < \infty), \qquad\quad \|f\|_{\scalebox{0.65}{$L_\infty(E; w(x)\, dx)$}} = \esssup_{x \in E} |f(x)|.
\]

The Whittaker function $W_{\alpha,\nu}(x)$ is the solution of Whittaker's differential equation ${d^2 u \over dx^2} + \bigl(-{1 \over 4} + {\alpha \over x} + {1/4 - \nu^2 \over x^2}\bigr) u = 0$ ($\alpha, \nu \in \mathbb{C}$) which is determined uniquely by the property
\begin{equation} \label{eq:prel_Whit_asyminfty}
W_{\alpha,\nu}(x) \sim x^\alpha e^{-{x \over 2}}, \qquad x \to \infty.
\end{equation}
For fixed $x$, the Whittaker $W$ function is an entire function of the first and the second parameter \cite[\S13.14(ii)]{dlmf}, and it admits the integral representation (cf.\ \cite{prudnikovI1986}, integral 2.3.6.9)
\begin{equation} \label{eq:prel_Whit_intrep}
W_{\alpha,\nu}(x) = {e^{-{x \over 2}} x^\alpha \over \Gamma({1 \over 2} - \alpha + \nu)} \int_0^\infty e^{-s} s^{-{1 \over 2} - \alpha + \nu} \Bigl(1 + {s \over x}\Bigr)^{\!-{1 \over 2} + \alpha + \nu\!} ds \qquad (\Re x > 0,\: \Re \alpha < \tfrac{1}{2} + \Re \nu)
\end{equation}
where $\Gamma(\cdot)$ is the Gamma function \cite[Chapter I]{erdelyiI1953}. The Whittaker W function is an even function of the parameter $\nu$ \cite[Equation 13.14.31]{dlmf}. For $\alpha \neq {1 \over 2} \pm \nu, {3 \over 2} \pm \nu, \ldots$, its asymptotic behavior near the origin is, cf.\ \cite[\S 13.14(iii)]{dlmf}
\begin{equation} \label{eq:prel_Whit_asymzero}
\begin{aligned}
& W_{\alpha,\nu}(x) = O\bigl(x^{{1 \over 2} - \Re\nu}\bigr) \qquad\: (\Re \nu \geq 0, \, \nu \neq 0), \\
& W_{\alpha,0}(x) = O\bigl(-x^{1 \over 2} \log x\bigr),
\end{aligned} \qquad\;\; x \to 0.
\end{equation}
The Whittaker function satisfies the recurrence relation \cite[Equation 13.15.13]{dlmf}
\begin{equation} \label{eq:prel_Whit_recurrence}
x^{1 \over 2} W_{\alpha + {1 \over 2},\nu + {1 \over 2}}(x) = (x + 2\nu) W_{\alpha,\nu}(x) + (\tfrac{1}{2} - \alpha - \nu) \, x^{1 \over 2} W_{\alpha - {1 \over 2},\nu - {1 \over 2}}(x)
\end{equation}
and it reduces to the modified Bessel function of the second kind (resp., to an elementary function) when the parameter $\alpha$ is equal to zero (resp., equal to ${1 \over 2} + \nu$) \cite[\S13.18(i), (iii)]{dlmf},
\begin{gather} 
\label{eq:prel_Whit_modbessel} W_{0,\nu}(2x) = \pi^{-{1 \over 2}} (2x)^{1 \over 2} K_\nu(x) \\
\label{eq:prel_Whit_elementary} W_{{1 \over 2} + \nu,\nu}(x) = x^{{1 \over 2}+\nu} e^{-x/2}.
\end{gather}
By \cite[Theorem 1.11]{yakubovich1996}, for $\alpha \in \mathbb{R}$ the asymptotic expansion of the Whittaker function with imaginary parameter $\nu = i\tau$ as $\tau \to \infty$ is
\begin{equation} \label{eq:prel_Whit_asymtau}
W_{\alpha,i\tau}(x) = (2x)^{1 \over 2} \tau^{\alpha - {1 \over 2}} e^{-\pi\tau/2} \cos\biggl( \tau \log\Bigl({x \over 4\tau}\Bigr) + {\pi \over 2}\Bigl({1 \over 2} - \alpha\Bigr) + \tau \biggr) \bigl[ 1 + O(\tau^{-1}) \bigr],
\end{equation}
the expansion being uniform in $0 < x \leq M$ ($M > 0$).

The confluent hypergeometric function of the second kind $\Psi(a,b;x)$, also known as the Tricomi function or the Kummer function of the second kind, can be defined via the Whittaker function as
\begin{equation} \label{eq:prel_confhyp_def}
\Psi(a,b;x) = e^{x/2} x^{-b/2\,} W_{{b \over 2}-a, {b \over 2} - {1 \over 2}}(x).
\end{equation}
This function is also commonly denoted by $U(a,b;x)$ \cite[Chapter 13]{dlmf}. A number of properties are obtained directly from this relation, such as the identity $x^{a+\nu} \Psi(a+\nu,1+2\nu;x) = x^{a-\nu} \Psi(a-\nu,1-2\nu;x)$ or the limiting forms of $\Psi(a,b;x)$ when $x \to \infty$ and $x \to 0$. We also note that the following differentiation formulas hold for $n \in \mathbb{N}$ \cite[Equations 6.6(12)--(13)]{erdelyiI1953}:
\begin{gather}
\label{eq:prel_confhyp_differentiate1} {d^n \over dx^n} \bigl[ x^{c-1} \Psi(a,c;x) \bigr] = (-1)^n \, (a-c+1)_n \, x^{c-n-1} \Psi(a,c-n;x) \\[1.5pt]
\label{eq:prel_confhyp_differentiate2} {d^n \over dx^n} \bigl[ x^{a+n-1} \Psi(a,c;x) \bigr] = (a)_n \, (a-c+1)_n \, x^{a-1} \Psi(a+n,c;x),
\end{gather}
where $(a)_n$ is the Pochhammer symbol, $(a)_0 = 1$ and $(a)_n = \prod_{j=0}^{n-1} (a+j)$ for $n \in \mathbb{N}$.

The parabolic cylinder function $D_\mu(z)$ is the solution of the differential equation ${d^2 u \over dz^2} + \bigl(\mu + {1 \over 2} - {z^2 \over 4} \bigr) u = 0$ which is given in terms of the Whittaker function by
\[
D_\mu(z) = 2^{{\mu \over 2} + {1 \over 4}} z^{-{1 \over 2}} W_{{{\mu \over 2} + {1 \over 4}}, {1 \over 4}}\bigl(\tfrac{z^2}{2}\bigr).
\]
It is an entire function of its parameter. An integral representation for this function is \cite[Equation 12.5.3]{dlmf}
\begin{equation} \label{eq:prel_parcyl_intrep}
D_\mu(z)  = {z^{\mskip 0.5\thinmuskip \mu} \, e^{-{z^2 \over 4}\,} \over \Gamma\bigl({1 \over 2}(1 - \mu)\bigr)} \int_0^\infty \! e^{-s} s^{-{1 \over 2}(1+\mu)} \biggl( 1+{2s \over z^2} \biggr)^{\!\!{\mu \over 2}} ds \qquad (\Re z > 0,\: \Re \mu < 1).
\end{equation}
The asymptotic form of $D_\mu(z)$ for large $z$ is \cite[Equation 8.4(1)]{erdelyiII1953}
\begin{equation} \label{eq:prel_parcyl_asyminfty}
D_\mu(z) \sim z^\mu e^{-{z^2 \over 4}} \qquad z \to \infty.
\end{equation}
The recurrence relation and differentiation formula for the parabolic cylinder function are \cite[Equations 8.2(14) and 8.2(16)]{erdelyiII1953}
\begin{align} \label{eq:prel_parcyl_recurrence}
D_{\nu + 1}(z) & = z D_\nu(z) - \nu D_{\nu-1}(z) \\
\label{eq:prel_parcyl_differentiate} {d^n \over dz^n}\bigl[ e^{-{z^2 \over 4}} D_\nu(z)  \bigr] & = (-1)^n e^{-{z^2 \over 4}} D_{\nu+n}(z) \qquad (n \in \mathbb{N})
\end{align}
and the parabolic cylinder function reduces to an exponential function when its parameter equals zero \cite[Equation 8.2(9)]{erdelyiII1953},
\begin{equation} \label{eq:prel_parcyl_expcase}
D_0(z) = e^{-{z^2 \over 4}}.
\end{equation}

\section{The product formula for the Whittaker function} \label{chap:Whit_prodform}

The main result of this paper, which will be proved in this section, is the following product formula for the Whittaker function of the second kind:

\begin{theorem} \label{thm:Whit_prodform}
The product $W_{\alpha,\nu}(x) W_{\alpha,\nu}(y)$ of two Whittaker functions of the second kind with different arguments admits the integral representation
\begin{equation} \label{eq:Whit_prodform}
W_{\alpha,\nu}(x) W_{\alpha,\nu}(y) = \int_0^\infty W_{\alpha,\nu}(\xi)\, k_\alpha(x,y,\xi)\, {d\xi \over \xi^2} \qquad (x,y > 0, \; \alpha, \nu \in \mathbb{C})
\end{equation}
where
\begin{align*}
k_\alpha(x,y,\xi) := & 2^{-1-\alpha} \pi^{-{1 \over 2}} (xy\xi)^{1 \over 2} \exp\biggl({x \over 2} + {y \over 2} + {\xi \over 2} - {(xy + x\xi + y\xi)^2 \over 8xy\xi}\biggr) D_{2\alpha}\biggl({xy + x\xi + y\xi \over (2xy\xi)^{1/2}}\biggr)
\end{align*}
being $D_\mu(z)$ the parabolic cylinder function.
\end{theorem}

We will prove this theorem through a sequence of lemmas, where we shall assume that $\alpha$ is a negative real number and $\nu$ is purely imaginary. In the final step of the proof, an analytic continuation argument will be used to remove this restriction.

The first lemma gives an alternative product formula which is less useful than \eqref{eq:Whit_prodform} because its kernel also depends on the second parameter of the Whittaker function.

\begin{lemma}
If $\alpha \in (-\infty, 0)$ and $\tau \in \mathbb{R}$, then the integral representation
\begin{equation} \label{eq:Whit_prodform_lem1}
\begin{aligned}
& W_{\alpha,i\tau}(x) W_{\alpha,i\tau}(p) \\
& \;\; = {(xp)^\alpha e^{-{x \over 2} - {p \over 2}} \over |\Gamma(\tfrac{1}{2}-\alpha + i\tau)|^2} \int_0^\infty\! \xi^{-1-\alpha} e^{-{\xi \over 2}} W_{\alpha,i\tau}(\xi) \!\int_0^\infty\!  w^{-2\alpha} \exp\biggl(-w - \Bigl({1 \over x}+{1 \over p}+{w \over xp}\Bigr) w\xi \biggr) dw\, d\xi
\end{aligned}
\end{equation}
is valid for $x, p > 0$.
\end{lemma}

\begin{proof}
From relation 2.21.2.17 in \cite{prudnikovIII1990} it follows that
\begin{align}
\nonumber & W_{\alpha,i\tau}(x) W_{\alpha,i\tau}(p) = (xp)^{{1 \over 2} - i\tau} e^{-{x \over 2} - {p \over 2}} \Psi\Bigl({1 \over 2} - \alpha - i\tau, 1-2i\tau; x\Bigr) \Psi\Bigl({1 \over 2} - \alpha - i\tau, 1-2i\tau; p\Bigr) \\
\nonumber & \;\; = {(xp)^{{1 \over 2} - i\tau} e^{-{x \over 2} - {p \over 2}} \over \Gamma(1-2\alpha)} \int_0^\infty\! e^{-w\,} w^{-2\alpha} \bigl[(w+x)(w+p)\bigr]^{-{1 \over 2} + \alpha + i\tau} \\
\nonumber & \hspace{0.2\textwidth}\: \times {}_2F_1\biggl({1 \over 2} - \alpha - i\tau,{1 \over 2} - \alpha - i\tau; 1-2\alpha;\, 1-{xp \over(w+x)(w+p)} \biggr) dw \\
\label{eq:Whit_prodform_pf1a} & \;\; = {(xp)^\alpha e^{-{x \over 2} - {p \over 2}} \over \Gamma(1-2\alpha)} \int_0^\infty\! e^{-w\,} w^{-2\alpha} {}_2F_1\biggl({1 \over 2} - \alpha - i\tau,{1 \over 2} - \alpha + i\tau; 1-2\alpha;\, -\Bigl({1 \over x} + {1 \over p}\Bigr) w - {w^2 \over px} \biggr) dw.
\end{align}
Here ${}_2F_1(a,b;c;z)$ is the Gauss hypergeometric function \cite[Chapter 15]{dlmf}; in the last step we used the transformation formula ${}_2F_1(a,b;c;z) = (1-z)^{-a} {}_2F_1\bigl(a,c-b;c;{z \over z-1}\bigr)$, cf.\ \cite[Equation 15.8.1]{dlmf}.

Next, according to integral 2.19.3.5 in \cite{prudnikovIII1990}, the Gauss hypergeometric function in \eqref{eq:Whit_prodform_pf1a} admits the integral representation
\begin{align*}
{}_2F_1\biggl( & {1 \over 2} - \alpha - i\tau,{1 \over 2} - \alpha + i\tau; 1-2\alpha; -\Bigl({1 \over x} + {1 \over p}\Bigr) w - {w^2 \over px} \biggr) \\
& \qquad = {\Gamma(1-2\alpha) \over |\Gamma(\tfrac{1}{2}-\alpha + i\tau)|^2} \!\int_0^\infty\!\! \xi^{-1-\alpha} \exp\biggl(-{\xi \over 2} -\Bigl({1 \over x} + {1 \over p} + {w \over xp} \Bigr) w\xi\biggr) W_{\alpha,i\tau}(\xi) d\xi
\end{align*}
and thus we have
\begin{equation} \label{eq:Whit_prodform_pf1b}
\begin{aligned}
& W_{\alpha,i\tau}(x) W_{\alpha,i\tau}(p) \\
& \;\; = {(xp)^\alpha e^{-{x \over 2} - {p \over 2}} \over|\Gamma(\tfrac{1}{2}-\alpha + i\tau)|^2} \int_0^\infty\! e^{-w\,} w^{-2\alpha} \!\int_0^\infty\!\! \xi^{-1-\alpha} \exp\biggl(-{\xi \over 2} - \Bigl({1 \over x} + {1 \over p} + {w \over xp} \Bigr) w\xi\biggr) W_{\alpha,i\tau}(\xi)\, d\xi\, dw.
\end{aligned}
\end{equation}
Using the assumption $\Re\alpha < 0$ and the limiting forms \eqref{eq:prel_Whit_asyminfty}, \eqref{eq:prel_Whit_asymzero} of the Whittaker function, we see that the integrals $\int_0^\infty\! e^{-w} w^{-2\alpha} \, dw$ and $\int_0^\infty\! \xi^{-1-\alpha} e^{-{\xi \over 2}} W_{\alpha,i\tau}(\xi) d\xi$ converge absolutely. Therefore, we can use Fubini's theorem to reverse the order of integration in \eqref{eq:Whit_prodform_pf1b}; doing so, we obtain \eqref{eq:Whit_prodform_lem1}.
\end{proof}

The previous lemma gives an integral representation for $|\Gamma(\tfrac{1}{2}-\alpha + i\tau)|^2 W_{\alpha,i\tau}(x) W_{\alpha,i\tau}(p)$ whose kernel does not depend on $\tau$. Integral representations for $|\Gamma(\tfrac{1}{2}-\alpha + i\tau)|^2 W_{\alpha,i\tau}(x)$ which share the same property are also known. In the next two lemmas we take advantage of these integral representations and of the uniqueness theorem for Laplace transforms in order to deduce that the product formula \eqref{eq:Whit_prodform} holds when $\alpha$ is a negative real number and $\nu = i\tau \in i \mathbb{R}$.

\begin{lemma}
The identity
\begin{equation} \label{eq:Whit_prodform_lem2}
\begin{aligned}
& 2^{2\alpha} x^{-\alpha} W_{\alpha,i\tau}(x)\! \int_0^\infty\! e^{-{s \over 2y}-{y \over 2}} y^{\alpha-2\,} W_{\alpha,i\tau}(y) dy \\
& \;\;\, = \! \int_0^\infty\! \Bigl(1+{2s \over x\xi}\Bigr)^{\!-{1 \over 2}} \biggl(\Bigl(1+{2s \over x\xi}\Bigr)^{1/2}+1\biggr)^{2\alpha} \exp\biggl[-\Bigl({x \over 2} + {\xi \over 2}\Bigr) \Bigl(1+{2s \over x\xi}\Bigr)^{1/2}\biggr] W_{\alpha,i\tau}(\xi)\, \xi^{\alpha-2} d\xi
\end{aligned}
\end{equation}
holds for $\alpha \in (-\infty, 0)$, $\tau \in \mathbb{R}$ and $x, s > 0$.
\end{lemma}

\begin{proof}
Using the change of variable $s=2w\xi(1+ {w \over x})$, we rewrite \eqref{eq:Whit_prodform_lem1} as
\begin{equation} \label{eq:Whit_prodform_pf2}
\begin{aligned}
& |\Gamma(\tfrac{1}{2}-\alpha + i\tau)|^2 W_{\alpha,i\tau}(x) W_{\alpha,i\tau}(p) \\
& \qquad = {1 \over 2} (xp)^\alpha e^{-{x \over 2} - {p \over 2}} \! \int_0^\infty\! e^{-{\xi \over 2}} \xi^{\alpha-2} W_{\alpha,i\tau}(\xi) \!\int_0^\infty\! e^{-{s \over 2p}} s^{-2\alpha} \Bigl(1+{2s\over x\xi}\Bigr)^{\!-{1 \over 2}} \biggl(\Bigl(1+{2s \over x\xi}\Bigr)^{\!{1 \over 2}}+1\biggr)^{2\alpha} \\
& \hspace{0.5\linewidth} \times \exp\biggl[\Bigl({x \over 2} + {\xi \over 2}\Bigr) \biggl(1-\Bigl(1+{2s\over x\xi}\Bigr)^{\!{1 \over 2}}\biggr)\biggr] ds\, d\xi \\
& \qquad = {1 \over 2} (xp)^\alpha e^{- {p \over 2}} \! \int_0^\infty\! e^{-{s \over 2p}} s^{-2\alpha} \!\int_0^\infty\! \Bigl(1+{2s\over x\xi}\Bigr)^{\!-{1 \over 2}} \biggl(\Bigl(1+{2s\over x\xi}\Bigr)^{\!{1 \over 2}}+1\biggr)^{2\alpha} \\
& \hspace{0.4\linewidth} \times \exp\biggl[-\Bigl({x \over 2} + {\xi \over 2}\Bigr) \Bigl(1+{2s\over x\xi}\Bigr)^{\!{1 \over 2}} \biggr]  W_{\alpha,i\tau}(\xi)\, \xi^{\alpha-2} d\xi \, ds
\end{aligned}
\end{equation}
where the absolute convergence of the iterated integral (see the proof of the previous lemma) justifies the change of order of integration.

On the other hand, by relation 2.19.5.18 in \cite{prudnikovIII1990} we have
\begin{equation} \label{eq:Whit_prodform_pf3}
\begin{aligned}
|\Gamma(\tfrac{1}{2}-\alpha + i\tau)|^2 W_{\alpha,i\tau}(p)
& = 2^{2\alpha - 1} \Gamma(1-2\alpha) p^\alpha e^{-{p \over 2}} \! \int_0^\infty \Bigl({1 \over 2y}+{1 \over 2p}\Bigr)^{\!-1+2\alpha} e^{-{y \over 2}} y^{\alpha-2\,} W_{\alpha,i\tau}(y)\, dy \\
& = 2^{2\alpha - 1} p^\alpha e^{-{p \over 2}} \int_0^\infty\!\!\! \int_0^\infty e^{- {s \over 2y} - {s \over 2p}} s^{-2\alpha} ds\, e^{-{y \over 2}} y^{\alpha-2\,} W_{\alpha,i\tau}(y)\, dy \\
& = 2^{2\alpha - 1} p^\alpha e^{-{p \over 2}} \! \int_0^\infty\! e^{-{s \over 2p}} s^{-2\alpha}\! \int_0^\infty\! e^{-{s \over 2y}-{y \over 2}} y^{\alpha-2\,} W_{\alpha,i\tau}(y)\, dy\, ds.
\end{aligned}
\end{equation}
Comparing \eqref{eq:Whit_prodform_pf2} and \eqref{eq:Whit_prodform_pf3}, and recalling the injectivity of Laplace transform, we deduce that \eqref{eq:Whit_prodform_lem2} holds.
\end{proof}

\begin{lemma}
The product formula \eqref{eq:Whit_prodform} holds for $\alpha < 0$, $\tau \in \mathbb{R}$ and $x, y > 0$.
\end{lemma}

\begin{proof}
We begin by deriving the following representation for the function of $s$ appearing in the right-hand side of \eqref{eq:Whit_prodform_lem2}:
\begin{align*}
& \Bigl(1+{2s \over x\xi}\Bigr)^{\!-{1 \over 2}} \biggl(\Bigl(1+{2s \over x\xi}\Bigr)^{1/2}+1\biggr)^{2\alpha} \exp\biggl[-\Bigl({x \over 2} + {\xi \over 2}\Bigr) \Bigl(1+{2s \over x\xi}\Bigr)^{\!{1 \over 2}}\biggr] \\
& = {1 \over \Gamma(-2\alpha)} \exp\biggl[-\Bigl({x \over 2} + {\xi \over 2}\Bigr) \Bigl(1+{2s \over x \xi}\Bigr)^{\!{1 \over 2}}\biggr] \int_0^\infty \exp\biggl(-u\Bigl(1+{2s \over x \xi}\Bigr)^{\!{1 \over 2}}\biggr) \gamma(-2\alpha,u)\, du \\
& = {(\pi x\xi)^{-{1 \over 2}} \over \Gamma(-2\alpha)} \! \int_0^\infty \Bigl(u + {x \over 2} + {\xi \over 2}\Bigr)\, \gamma(-2\alpha,u) \int_0^\infty\!\! y^{-{1 \over 2}} \exp\biggl[ -\bigl( 2s+x\xi \bigr) {1 \over 4y} - \Bigl(u + {x \over 2} + {\xi \over 2}\Bigr)^{\!2}  {y \over x\xi} \biggr] dy\, du \\
& = {(\pi x\xi)^{-{1 \over 2}} \over \Gamma(-2\alpha)} \! \int_0^\infty\!\! e^{-{s \over 2y}} \exp\Bigl(-{x \xi \over 4y}\Bigr)\, y^{-{1 \over 2}}\!\! \int_0^\infty \Bigl(u + {x \over 2} + {\xi \over 2}\Bigr)\, \exp\biggl( - \Bigl(u + {x \over 2} + {\xi \over 2}\Bigr)^{\!2} {y \over x \xi} \biggr)\, \gamma(-2\alpha,u) du\, dy
\end{align*}
where $\gamma(\cdot, \cdot)$ is the incomplete Gamma function \cite[Chapter IX]{erdelyiII1953}. In the first two equalities we have used integral 8.14.1 in \cite{dlmf} and integral 2.3.16.3 in \cite{prudnikovI1986}, respectively, and the positivity of the integrand allows us to change the order of integration. Substituting in \eqref{eq:Whit_prodform_lem2}, we find that
\begin{equation} \label{eq:Whit_prodform_pf4}
\begin{aligned}
& \Gamma(-2\alpha) 2^{{1 \over 2} + 2\alpha} \pi^{-{1 \over 2}} x^{{1 \over 2}-\alpha} W_{\alpha,i\tau}(x)\! \int_0^\infty\! e^{-{s \over 2y}-{y \over 2}} y^{\alpha-2\,} W_{\alpha,i\tau}(y) dy \\
& \quad\; = \! \int_0^\infty\!\xi^{-{5 \over 2}+\alpha\,} W_{\alpha,i\tau}(\xi) \!\int_0^\infty\!\! e^{-{s \over 2y}} \exp\Bigl(-{x \xi \over 4y}\Bigr) y^{-{1 \over 2}} \\
& \hspace{0.15\linewidth} \times  \int_0^\infty \Bigl(u + {x \over 2} + {\xi \over 2}\Bigr)\, \exp\biggl( - \Bigl(u + {x \over 2} + {\xi \over 2}\Bigr)^{\!2} {y \over x \xi} \biggr)\, \gamma(-2\alpha,u) du\, dy\, d\xi \\
& \quad\; = \! \int_0^\infty\!\! e^{-{s \over 2y}} y^{-{1 \over 2}} \!\! \int_0^\infty\!\! \xi^{-{5 \over 2}+\alpha} \exp\Bigl(-{x \xi \over 4y}\Bigr) W_{\alpha,i\tau}(\xi) \\
& \hspace{0.15\linewidth} \times \int_0^\infty \Bigl(u + {x \over 2} + {\xi \over 2}\Bigr)\, \exp\biggl( - \Bigl(u + {x \over 2} + {\xi \over 2}\Bigr)^{\!2} {y \over x \xi} \biggr)\, \gamma(-2\alpha,u) du\, d\xi\, dy
\end{aligned}
\end{equation}
where the order of integration can be interchanged because of the absolute convergence of the triple integral, which follows from the inequality $\gamma(-2\alpha,u) \leq \Gamma(-2\alpha)$ and the equalities
\begin{align*}
& \int_0^\infty\! \xi^{-{5 \over 2}+\alpha} \bigl|W_{\alpha,i\tau}(\xi)\bigr| \! \int_0^\infty\! e^{-{s \over 2y}} y^{-{1 \over 2}} \exp\Bigl(-{x\xi \over 2y}\Bigr) \! \int_0^\infty \Bigl(u + {x \over 2} + {\xi \over 2}\Bigr)\, \exp\biggl( - \Bigl(u + {x \over 2} + {\xi \over 2}\Bigr)^{\!2} {y \over x \xi} \biggr) du\, dy\, d\xi \\
& \quad = {x \over 2} \int_0^\infty\!\! \xi^{-{5 \over 2}+\alpha} \bigl|W_{\alpha,i\tau}(\xi)\bigr| \!\int_0^\infty\!\! \exp\biggl(-{s \over 2y} - {x \xi \over 4y} - {y \over 2} - {xy \over 4\xi} - {\xi y \over 4x}\biggr) y^{-{3 \over 2}} dy\, d\xi \\
& \quad = 2^{-{1 \over 2}} (\pi x)^{1 \over 2} \! \int_0^\infty\! \xi^{-3+\alpha} \Bigl(1+{2s \over x\xi}\Bigr)^{\!-{1 \over 2}} \exp\biggl(- \Bigl({x \over 2}+{\xi \over 2}\Bigr) \Bigl(1+{2s \over x\xi}\Bigr)^{\!{1 \over 2}} \biggr) \bigl|W_{\alpha,i\tau}(\xi)\bigr|\, d\xi < \infty 
\end{align*}
(which follow from integral 2.3.16.3 in \cite{prudnikovI1986} and straighforward calculations; the convergence of the latter integral can be verified using the limiting forms \eqref{eq:prel_Whit_asyminfty}, \eqref{eq:prel_Whit_asymzero} of the Whittaker function).

Using, as in the previous proof, the injectivity of Laplace transform, from \eqref{eq:Whit_prodform_pf4} it follows that
\begin{equation} \label{eq:Whit_prodform_pf5}
\begin{aligned}
W_{\alpha,i\tau}(x)  W_{\alpha,i\tau}(y) & = {2^{- 2\alpha} \pi^{-{1 \over 2}} \over \Gamma(-2\alpha)} x^{-{1 \over 2}+\alpha} y^{{3 \over 2}-\alpha} e^{y \over 2} \! \int_0^\infty\!\! \xi^{-{5 \over 2}+\alpha} \exp\Bigl(-{x \xi \over 4y}\Bigr) W_{\alpha,i\tau}(\xi) \\
& \qquad\qquad \times \int_0^\infty \Bigl(u + {x \over 2} + {\xi \over 2}\Bigr)\, \exp\biggl( - \Bigl(u + {x \over 2} + {\xi \over 2}\Bigr)^{\!2} {y \over x \xi} \biggr)\, \gamma(-2\alpha,u) du\, d\xi.
\end{aligned}
\end{equation}
Let us compute the inner integral. Since ${d \over du} \gamma(-2\alpha,u) = u^{-1-2\alpha} e^{-u}$ and
\[
\int\Bigl(u + {x \over 2} + {\xi \over 2}\Bigr)\, \exp\biggl( - \Bigl(u + {x \over 2} + {y \over 2}\Bigr)^{\!2} {y \over x\xi} \biggr)du = - {x\xi \over 2y} \exp\biggl(- \Bigl( u + {x \over 2} + {\xi \over 2} \Bigr)^{\!2} {y \over x\xi} \biggr),
\]
we obtain, using integration by parts,
\begin{equation} \label{eq:Whit_prodform_pf6}
\begin{aligned}
& \int_0^\infty \Bigl(u + {x \over 2} + {\xi \over 2}\Bigr)\, \exp\biggl( - \Bigl(u + {x \over 2} + {\xi \over 2}\Bigr)^{\!2} {y \over x \xi} \biggr)\, \gamma(-2\alpha,u) du \\
& \quad = {x\xi \over 2y} \int_0^\infty u^{-1-2\alpha} e^{-u} \exp\biggl(- \Bigl( u + {x \over 2} + {\xi \over 2} \Bigr)^{\!2} {y \over x\xi} \biggr) du \\
& \quad = \Gamma(-2\alpha) \Bigl({x \xi \over 2y}\Bigr)^{\!1-\alpha} \exp\Bigl( {x \over 4} + {\xi \over 4} - {y \over 4} + {x \xi \over 8y} - {xy \over 8\xi} - {y \xi \over 8x} \Bigr) D_{2\alpha}\biggl({xy + x\xi + y\xi \over (2xy\xi)^{1/2}}\biggr)
\end{aligned}
\end{equation}
where we applied relation 2.3.15.3 in \cite{prudnikovI1986}. Substituting this in \eqref{eq:Whit_prodform_pf5}, we conclude that \eqref{eq:Whit_prodform} holds for all $\alpha < 0$ and  $\nu = i\tau \in i \mathbb{R}$.
\end{proof}

\begin{proof}[Proof of Theorem \ref{thm:Whit_prodform}]
To simplify the notation, throughout the proof we write $f_{\alpha,\nu}(t) := t^{-\alpha} W_{\alpha,\nu}(t)$. We use an analytic continuation argument to extend the identity \eqref{eq:Whit_prodform} to all $\alpha, \nu \in \mathbb{C}$. To that end, let us prove that the right-hand side of \eqref{eq:Whit_prodform} is an entire function of each of the variables $\alpha$ and $\nu$. Let $M > 0$ and suppose that ${1 \over M} \leq {1 \over 2} - \Re\alpha \leq M$ and $0 \leq \Re\nu \leq M$. Then for $t>0$ we have
\begin{align*}
\bigl|f_{\alpha,\nu}(t)\bigr| & = \bigl|f_{\alpha,-\nu}(t)\bigr| = {e^{-{t \over 2}} \over |\Gamma({1 \over 2} - \alpha + \nu)|} \biggl| \int_0^\infty e^{-s} s^{-{1 \over 2} - \alpha + \nu} \Bigl(1 + {s \over t}\Bigr)^{\!-{1 \over 2} + \alpha + \nu} ds \biggr| \\
& \qquad\;\; \leq {e^{-{t \over 2}} \over |\Gamma({1 \over 2} - \alpha + \nu)|} \int_0^\infty e^{-s} s^{-1} (s^{1/M} + s^{2 M}) \Bigl(1 + {s \over t}\Bigr)^{\!M} ds \\
& \qquad\;\; = {1 \over |\Gamma({1 \over 2} - \alpha + \nu)|}\Bigl[\Gamma\bigl(\tfrac{1}{M}\bigr) f_{{1 \over 2}(M - {1 \over M} + 1),{1 \over 2}(M + {1 \over M})}(t)  + \Gamma(2M)  f_{{1 \over 2}(1-M),{3M \over 2}}(t) \Bigr]
\end{align*}
where we have used the integral representation \eqref{eq:prel_Whit_intrep}. Moreover, letting $n \in \mathbb{N}$, a repeated application of the recurrence relation \eqref{eq:prel_Whit_recurrence} shows that $f_{\alpha+{n\over 2},\nu+{n\over 2}}(t) = p_{n,\alpha,\nu}^{(1)}\smash{\bigl({1 \over t}\bigr)} f_{\alpha,\nu}(t) + p_{n,\alpha,\nu}^{(2)}\smash{\bigl({1 \over t}\bigr)} f_{\alpha-{1 \over 2},\nu-{1\over 2}}(t)$, where the $p_{n,\alpha,\nu}^{(i)}(\cdot)$ are polynomials of degree at most $n$ whose coefficients depend on $\alpha$ and $\nu$. Therefore, for ${1 \over M} \leq {1 \over 2} - \Re\alpha \leq M - {1 \over 2}$ and $-M + {1 \over 2} \leq \Re\nu \leq M$ we have
\begin{equation} \label{eq:Whit_prodform_pf8}
\begin{aligned}
& \bigl|f_{\alpha+{n\over 2},\nu+{n\over 2}}(t)\bigr| \leq \bigl|p_{n,\alpha,\nu}^{(1)}\bigl(\tfrac{1}{t}\bigr) f_{\alpha,\nu}(t)\bigr| + \bigl|p_{n,\alpha,\nu}^{(2)}\bigl(\tfrac{1}{t}\bigr) f_{\alpha-{1 \over 2},\nu-{1\over 2}}(t)\bigr| \\
& \qquad \leq \bigl({\bigl|p_{n,\alpha,\nu}^{(1)}\bigl(\tfrac{1}{t}\bigr)\bigr|+\bigl|p_{n,\alpha,\nu}^{(2)}\bigl(\tfrac{1}{t}\bigr)\bigr|}\bigr)\, G(\alpha,\nu) \Bigl[\Gamma\bigl(\tfrac{1}{M}\bigr) f_{{1 \over 2}(M - {1 \over M} + 1),{1 \over 2}(M + {1 \over M})}(t)  + \Gamma(2M)  f_{{1 \over 2}(1-M),{3M \over 2}}(t) \Bigr]
\end{aligned}
\end{equation}
where 
\[
G(\alpha,\nu) = \begin{cases}2|\Gamma({1 \over 2} - \alpha + \nu)|^{-1}, & {1 \over 2} \leq  \Re\nu \leq M \\
|\Gamma({1 \over 2} - \alpha + \nu)|^{-1} + |\Gamma({3 \over 2} - \alpha - \nu)|^{-1}, & 0 \leq \Re\nu < {1 \over 2} \\
|\Gamma({1 \over 2} - \alpha - \nu)|^{-1} + |\Gamma({3 \over 2} - \alpha - \nu)|^{-1}, & -M + {1 \over 2} \leq \Re \nu < 0.
\end{cases}
\]
Similarly, for ${1 \over M} \leq {1 \over 2} - \Re\alpha \leq M$ the integral representation \eqref{eq:prel_parcyl_intrep} gives
\begin{align*}
\bigl|D_{2\alpha}(t)\bigr| & = {e^{-{t^2 \over 4}\,} \over |\Gamma({1 \over 2} - \alpha)|} t^{2\, \Re\alpha} \biggl| \int_0^\infty e^{-s} s^{-{1 \over 2} - \alpha} \Bigl( 1+{2s \over t^2} \Bigr)^{\!\alpha} ds \biggr| \\
& \leq {e^{-{t^2 \over 4}} t \over |\Gamma({1 \over 2} - \alpha)|} (t^{-2M} + t^{-{2 \over M}\!}) \int_0^\infty e^{-s} s^{-1} (s^{{1 \over M}\!} + s^M) \Bigl( 1+{2s \over t^2} \Bigr)^{\!{1 \over 2} - {1 \over M}} ds \\
& = {t \over |\Gamma({1 \over 2} - \alpha)|} (t^{-2M} + t^{-{2 \over M}\!})\, \Bigl[ \Gamma\bigl(\tfrac{1}{M}\bigr) f_{{3 \over 4} - {1 \over M},{1 \over 4}}\bigl(\tfrac{t^2}{2}\bigr) + \Gamma(M) f_{{3 \over 4} - {1 \over 2}(M+{1 \over M}),{1 \over 4}+{1 \over 2}(M+{1 \over M})}\bigl(\tfrac{t^2}{2}\bigr) \Bigr] \end{align*}
and, by \eqref{eq:prel_parcyl_recurrence}, for each $n \in \mathbb{N}$ we have $D_{2\alpha + n}(t) = p_{n,\alpha}^{(3)}(t) D_{2\alpha}(t) + p_{n,\alpha}^{(4)}(t) D_{2\alpha-1}(t)$, being $p_{n,\alpha}^{(j)}(\cdot)$ polynomials of degree at most $n$ with coefficients depending on $\alpha$, hence
\begin{equation} \label{eq:Whit_prodform_pf9}
\begin{aligned}
\bigl|D_{2\alpha+n}(t)\bigr| & \leq \bigl(|\Gamma(\tfrac{1}{2} - \alpha)|^{-1\!} + |\Gamma(1 - \alpha)|^{-1}\bigr) \bigl({|p_{n,\alpha,\nu}^{(3)}(t)|+|p_{n,\alpha}^{(4)}(t)|}\bigr) \, t \, (t^{-2M} + t^{-{2 \over M}\!}) \\
& \hspace{0.2\linewidth} \times \Bigl[ \Gamma\bigl(\tfrac{1}{M}\bigr) f_{{3 \over 4} - {1 \over M},{1 \over 4}}\bigl(\tfrac{t^2}{2}\bigr) + \Gamma(M) f_{{3 \over 4} - {1 \over 2}(M+{1 \over M}),{1 \over 4}+{1 \over 2}(M+{1 \over M})}\bigl(\tfrac{t^2}{2}\bigr) \Bigr].
\end{aligned}
\end{equation}
Using the inequalities \eqref{eq:Whit_prodform_pf8}, \eqref{eq:Whit_prodform_pf9} and the limiting forms \eqref{eq:prel_Whit_asyminfty}, \eqref{eq:prel_Whit_asymzero} for the Whittaker function, one can verify without difficulty that 
\begin{equation} \label{eq:Whit_prodform_pf10}
\sup_{(\alpha,\nu) \in \mathcal{R}_{\mathsmaller{M}}} \int_0^\infty \Bigl|W_{\alpha+{n \over 2},\nu+{n \over 2}}(\xi)\, k_{\alpha+{n \over 2}}(x,y,\xi)\Bigr| {d\xi \over \xi^2} < \infty
\end{equation}
where $\mathcal{R}_{\mathsmaller{M}} = \bigl\{ (\alpha,\nu): {1 \over M} \leq {1 \over 2} - \Re\alpha \leq M - {1 \over 2},\; -M + {1 \over 2} \leq \Re\nu \leq M \bigr\}$. Since $M$ and $n$ are arbitrary, the known results on the analyticity of parameter-dependent integrals (e.g.\ \cite{mattner2001}) yield that $\int_0^\infty W_{\alpha,\nu}(\xi)\, k_\alpha(x,y,\xi)\, \xi^{-2\alpha} e^{-{1 \over 2\xi}}d\xi$ is an entire function of the parameter $\alpha$ and the parameter $\nu$. As the left-hand side of \eqref{eq:Whit_prodform} is also an entire function of $\alpha$ and $\nu$, by analytic continuation we conclude that the product formula \eqref{eq:Whit_prodform} extends to all $\alpha, \nu \in \mathbb{C}$, as we wanted to show.
\end{proof}

\begin{remark} \label{rmk:Whit_prodform_obs}
\textbf{(a)} The product formula \eqref{eq:Whit_prodform} can be equivalently written in terms of the confluent hypergeometric function of the second kind as
\begin{equation} \label{eq:confhyp_prodform}
\begin{aligned}
& (xy)^{a+\nu} \Psi(a+\nu,1+2\nu;x) \Psi(a+\nu,1+2\nu;y) = \\
& \quad\;\; = \int_0^\infty \xi^{a+\nu} \Psi(a+\nu,1+2\nu;\xi) \, q_a(x,y,\xi) \, \ma(\xi) d\xi & \qquad\quad (x,y > 0, \; a, \nu \in \mathbb{C}),
\end{aligned}
\end{equation}
being
\begin{align}
\nonumber & \ma(\xi) := \xi^{-2a-1} e^{-\xi}, \\
\label{eq:confhyp_prodform_kerndef} & \begin{aligned}
q_a(x,y,\xi) & := (xy\xi)^{a-{1 \over 2}} e^{{1 \over 2}(x+y+\xi)} k_{{1 \over 2} - a}(x,y,\xi) \\[2pt]
& \: = 2^{-{3 \over 2}+a} \pi^{-{1 \over 2}} (xy\xi)^{a} \exp\biggl( x + y + \xi - {(xy + x\xi + y\xi)^2 \over 8xy\xi} \biggr) D_{1-2a}\biggl( {xy + x\xi + y\xi \over (2xy\xi)^{1/2}} \biggr).
\end{aligned}
\end{align}
 \\[-15pt]

\noindent\textbf{(b)} It follows from \eqref{eq:prel_Whit_modbessel} and \eqref{eq:prel_parcyl_expcase} that in the particular case $\alpha = 0$, \eqref{eq:Whit_prodform} specializes into
\[
K_\nu(x) K_\nu(y) = {1 \over 2} \int_0^\infty K_\nu(\xi) \exp\biggl( -{xy \over 2\xi} - {x \xi \over 2y} - {y \xi \over 2x} \biggr) {d\xi \over \xi}
\]
which is the Macdonald formula for the product of modified Bessel functions (cf.\ \cite[\S7.7.6]{erdelyiII1953} and \cite[Equation (1.103)]{yakubovich1996}). \\[-10pt]

\noindent\textbf{(c)} Since the parabolic cylinder function $D_\nu(t)$ is a positive function of $t > 0$ whenever $\nu \in (-\infty,1]$ (as can be seen e.g.\ from the representation \eqref{eq:prel_parcyl_intrep}), we have
\[
q_a(x,y,\xi) > 0 \qquad \text{for all } a \geq 0 \text{ and } x, y, \xi > 0.
\]
This positivity property means that, for $a \geq 0$, the convolution operator induced by the product formula \eqref{eq:confhyp_prodform} (cf.\ Section \ref{chap:Whit_convol}) is positivity-preserving. \\[-10pt]

\noindent\textbf{(d)} Useful upper bounds for the kernels of the product formulas \eqref{eq:Whit_prodform} and \eqref{eq:confhyp_prodform} are the following:
\begin{gather} 
\nonumber \bigl|k_\alpha(x,y,\xi)\bigr| \leq A(y)\, (xy\xi)^{{1 \over 2}-\alpha} (xy+x\xi+y\xi)^{2\alpha} \exp\biggl( -{xy \over 4\xi} - {x \xi \over 4y} - {y \xi \over 4x} \biggr) \qquad (x, y, \xi > 0, \; \alpha \in \mathbb{R}) \\
\label{eq:confhyp_prodformkern_bound} \bigl|q_a(x,y,\xi)\bigr| \leq A(y)\, (xy\xi)^{2a-{1 \over 2}} (xy+x\xi+y\xi)^{1-2a} \exp\biggl( \xi - {(x(\xi-y) +y\xi)^2 \over 4xy\xi} \biggr) \qquad (x, y, \xi > 0, \; a \in \mathbb{R}), 
\end{gather}
where
\[
A(y) = 2^{-\alpha-1} \pi^{-{1 \over 2}} \ccdot \Bigl( \max_{t \geq y^{1/2}} t^{-2\alpha} e^{t^2 \over 4} D_\alpha(t) \Bigr) < \infty \qquad (y > 0)
\]
(in the second upper bound, we replace $\alpha$ by ${1 \over 2} - a$ in the expression of $A(y)$). These equivalent upper bounds follow from the inequality ${(xy+x\xi+y\xi)^2 \over 2xy\xi} \geq y$ and the fact that, by \eqref{eq:prel_parcyl_asyminfty}, the function $t^{-2\alpha} e^{t^{2\!}/4} D_{2\alpha}(t)$ is bounded on the interval $ [y^{1/2},\infty)$.
\end{remark}

\section{The convolution associated with the index Whittaker transform} \label{chap:Whit_convol}

\subsection{Generalized translations} \label{sec:Whit_convol_gentransl}

We start this section by defining the generalized translation induced by the Whittaker product formula:

\begin{definition} \label{def:confhyp_gentransl}
Let $1 \leq p \leq \infty$ and $a \geq 0$. The linear operator
\[
(\mathcal{T}_a^y f)(x) = \int_0^\infty \! f(\xi) q_a(x,y,\xi)\, \ma(\xi) d\xi \qquad \bigl( f \in L_p\bigl((0,\infty); \ma(x) dx\bigr),\; x, y > 0 \bigr)
\]
where $\ma(\xi) = \xi^{-2a-1} e^{-\xi}$ and $q_a(x,y,\xi)$ is defined by \eqref{eq:confhyp_prodform_kerndef}, will be called the \emph{index Whittaker translation operator (of order $a$)}.
\end{definition}

Observe that the operator $\mathcal{T}_a^y$ was defined so that \eqref{eq:confhyp_prodform} reads
\[
\Bigl(\mathcal{T}_a^y \bigl[\xi^{a+\nu} \Psi(a+\nu,1+2\nu;\xi)\bigr] \! \Bigr)\!(x) = (xy)^{a+\nu} \Psi(a+\nu,1+2\nu;x) \Psi(a+\nu,1+2\nu;y),
\]
meaning that we have chosen the confluent hypergeometric form of the product formula for constructing the generalized translation operator. As we will see, this choice turns out to be particularly convenient for studying the $L_p$ properties of the corresponding convolution operator.

The following lemma gives the closed-form expression for the index Whittaker translation of the power function $\theta(x) = x^\beta$.

\begin{lemma} \label{lem:confhyp_gentransl_power}
For $a, \beta \in \mathbb{C}$, we have
\begin{equation} \label{eq:confhyp_gentransl_power}
\int_0^\infty\! \xi^\beta q_a(x,y,\xi)\, \ma(\xi) d\xi = (xy)^\beta \Psi(\beta, 1-2a+2\beta; x+y) \qquad (x, y > 0).
\end{equation}
In particular, $\int_0^\infty\! q_a(x,y,\xi)\, \ma(\xi) d\xi = 1$ for $a \in \mathbb{C}$ and $x, y > 0$.
\end{lemma}

\begin{proof}
Fix $x, y > 0$, and suppose that $a > {1 \over 2}$ and $\beta \in \mathbb{R}$. Using the definition \eqref{eq:confhyp_prodform_kerndef} and the integral representation of $D_{1-2a}\bigl( {xy + x\xi + y\xi \over (2xy\xi)^{1/2}} \bigr)$ obtained by exchanging the variables $y$ and $\xi$ in \eqref{eq:Whit_prodform_pf6}, we find that for each $a > {1 \over 2}$ we have
\begin{align*}
& q_a(x,y,\xi)\, \xi^{-2a-1} e^{-\xi} = \\
& \;\; = {2^{2a-2} \pi^{-{1 \over 2}} \over \Gamma(2a-1)} (x y)^{1 \over 2} \xi^{-{3 \over 2}} \exp\Bigl( {x \over 2} + {y \over 2} - {xy \over 4\xi} \Bigr) \! \int_0^\infty \! u^{2a-2} \exp\biggl( -u -\Bigl(u + {x \over 2} + {y \over 2}\Bigr)^{\!2} {\xi \over xy} \biggr) du.
\end{align*}
Consequently, we may compute
\begin{align*}
& \int_0^\infty\! \xi^\beta q_a(x,y,\xi)\, \ma(\xi) d\xi \\
& \quad = {2^{2a-2} \pi^{-{1 \over 2}} \over \Gamma(2a-1)} (x y)^{1 \over 2} \exp\Bigl( {x \over 2} + {y \over 2} \Bigr) \! \int_0^\infty \! u^{2a-2} e^{-u} \! \int_0^\infty \! \xi^{\beta-{3 \over 2}} \exp\biggl( -\Bigl(u + {x \over 2} + {y \over 2}\Bigr)^{\!2} {\xi \over xy} - {xy \over 4\xi} \biggr) d\xi\, du \\
& \quad = {2^{2a-\beta-{1 \over 2}} \over \Gamma(2a-1)} \pi^{-{1 \over 2}} (xy)^\beta \exp\Bigl( {x \over 2} + {y \over 2} \Bigr) \int_0^\infty u^{2a-2} \Bigl( u + {x \over 2} + {y \over 2} \Bigr)^{{1 \over 2} - \beta} e^{-u} K_{\beta - {1 \over 2}}\Bigl( u + {x \over 2} + {y \over 2} \Bigr)\, du \\
& \quad = {2^{2a-\beta-{1 \over 2}} \over \Gamma(2a-1)} \pi^{-{1 \over 2}} (xy)^\beta \exp(x+y) \int_{{x \over 2}+{y \over 2}}^\infty\! t^{{1 \over 2} - \beta}  \Bigl( t - {x \over 2} - {y \over 2} \Bigr)^{2a-2} e^{-t} K_{\beta - {1 \over 2}}(t)\, dt \\
& \quad = (xy)^\beta \Psi(\beta, 1-2a+2\beta; x+y)
\end{align*}
where the first equality is obtained by changing the order of integration (note the positivity of the integrand), the second equality follows from integral 2.3.16.1 in \cite{prudnikovI1986} and a few simplifications, the third equality results from the change of variables $u = t - {x \over 2} - {y \over 2}$, and the last equality uses relation 2.16.7.5 in \cite{prudnikovII1986}. This proves that \eqref{eq:confhyp_gentransl_power} holds in the case $a > {1 \over 2}$ and $\beta \in \mathbb{R}$.

To extend the result to all $a, \beta \in \mathbb{C}$, we can use an analytic continuation argument similar to that of the proof of Theorem \ref{thm:Whit_prodform}. Indeed, using \eqref{eq:Whit_prodform_pf9} and the elementary inequality $|\xi^\beta| \leq \xi^{-M} + \xi^M$ ($\xi > 0$, $\beta \in [-M,M]$) one can verify, as in the previous proof, that
\[
\sup_{(a,\beta) \in \bar{\mathcal{R}}_{\mathsmaller{M}}} \int_0^\infty \bigl|\, \xi^{\beta} q_{a-{n \over 2}}(x,y,\xi)\, \xi^{-2a + n -1}  \bigr| e^{-\xi} d\xi < \infty
\]
where $\bar{\mathcal{R}}_{\mathsmaller{M}} = \bigl\{ (a,\beta): {1 \over M} \leq \Re a \leq M - {1 \over 2},\, -M \leq \Re \beta \leq M \bigr\}$, being $M>0$ and $n \in \mathbb{N}$ arbitrary. Both sides of \eqref{eq:confhyp_gentransl_power} are therefore entire functions of the parameter $a$ and the parameter $\beta$; consequently, the principle of analytic continuation gives \eqref{eq:confhyp_gentransl_power} in the general case. By \eqref{eq:prel_Whit_elementary}, the right-hand side of \eqref{eq:confhyp_gentransl_power} equals $1$ when $\beta = 0$.
\end{proof}

The next proposition gives the basic continuity and $L_p$ properties of the index Whittaker translation operator. We consider the weighted $L_p$ spaces
\[
L_p^a := L_p\bigl((0,\infty); \ma(x) dx\bigr) \qquad\;\; (1 \leq p \leq \infty, \: 0 \leq a < \infty)
\]
with the usual norms
\[
\|f\|_{p,a} = \biggl(\int_0^\infty |f(x)|^p \ma(x) dx\biggr)^{1/p} \quad (1 \leq p < \infty), \qquad\quad \|f\|_{\infty} \equiv \|f\|_{\infty, a\!} = \esssup_{0 < x < \infty} |f(x)|.
\]

\begin{proposition} \label{prop:confhyp_gentransl_props}
Fix $a \geq 0$ and $y > 0$. Then: \\[-8pt]

\textbf{(a)} If $f \in L_\infty^a$ is such that $0 \leq f \leq 1$, then $0 \leq \mathcal{T}_a^y f \leq 1$; \\[-8pt]

\textbf{(b)} For each $1 \leq p \leq \infty$, we have
\[
\|\mathcal{T}_a^y f\|_{p,a} \leq \|f\|_{p,a}  \qquad \text{ for all } f \in L_p^a
\]
(in particular, $\mathcal{T}_a^y \bigl(L_p^a\bigr) \subset L_p^a$); \\[-8pt]

\textbf{(c)} If $f \in L_p^a$ where $1 < p \leq \infty$, then $\mathcal{T}_a^y f \in \mathrm{C}(0,\infty)$, and for $1 < p < \infty$ we also have
\[
\lim_{h \to 0} \|\mathcal{T}_a^{y+h} f - \mathcal{T}_a^y f \|_{p,a} = 0;
\]

\textbf{(d)} If $f \in \mathrm{C}_\mathrm{b}(0,\infty)$, then $(\mathcal{T}_a^y f)(x) \to f(y)$ as $x \to \infty$; \\[-8pt]

\textbf{(e)} If $f \in L_\infty^a$ is such that $\lim_{x \to 0} f(x) = 0$, then $\lim_{x \to 0} (\mathcal{T}_a^y f)(x) = 0$.
\end{proposition}

\begin{proof}
Throughout this proof the letter $C$ stands for a constant whose exact value may change from line to line.

\textbf{(a)} By Lemma \ref{lem:confhyp_gentransl_power}, if $f \equiv 1$ then $\mathcal{T}_a^y f \equiv 1$. Moreover, Remark \ref{rmk:Whit_prodform_obs}(c) means that $\mathcal{T}_a^y f$ is nonnegative whenever $f$ is nonnegative. Recalling that $\mathcal{T}_a^y$ is a linear operator, we see that we have $0 \leq \mathcal{T}_a^y f \leq 1$ whenever $0 \leq f \leq 1$. \\[-8pt]

\textbf{(b)} The case $p = \infty$ was proved in part (a). Now, for $1 \leq p < \infty$ and $f \in L_p^a$ we have
\begin{align*}
\| \mathcal{T}_a^y f \|_{p,a}^p & = \int_0^\infty \biggl| \int_0^\infty \! f(\xi) q_a(x,y,\xi)\, \ma(\xi) d\xi \biggr|^p \ma(x) dx \\
& \leq \int_0^\infty\! \int_0^\infty \! |f(\xi)|^p q_a(x,y,\xi)\, \ma(\xi) d\xi \, \ma(x) dx \\
& = \int_0^\infty\! \int_0^\infty \! q_a(x,y,\xi)\, \ma(x) dx \, |f(\xi)|^p \ma(\xi) d\xi = \| f \|_{p,a}^p
\end{align*}
where we have used the final statement in Lemma \ref{lem:confhyp_gentransl_power}, the fact that $q_a(x,y,\xi)$ is positive and symmetric, and H\"{o}lder's inequality. \\[-8pt]

\textbf{(c)} For $f \in L_p^a$ ($1<p<\infty$), by Young's inequality we have
\[
\int_0^\infty \! |f(\xi)| q_a(x,y,\xi)\, \ma(\xi) d\xi \leq {1 \over p} \|f\|_{p,a}^p + {1 \over q} \int_0^\infty \! |q_a(x,y,\xi)|^q \ma(\xi) d\xi
\]
and therefore the continuity of $\mathcal{T}_a^y f$ will be proved if we show that, for each $1 \leq q < \infty$, the integral $\int_0^\infty \! |q_a(x,y,\xi)|^q\, \ma(\xi) d\xi$ converges absolutely and locally uniformly. In fact, let us fix $M>0$; then,
\begin{align} 
\nonumber q_a(x,y,\xi) & \leq A(y)\, (xy\xi)^{1 \over 2} \Bigl(1 + {y \over x} + {y \over \xi}\Bigr) \exp\biggl( {x \over 2} + {y \over 2} - {xy \over 4\xi} - {(x-y)^2 \xi \over 4xy} \biggr) \\
\label{eq:confhyp_gentransl_cont_pf1} & \leq A_1(y) \, \xi^{1 \over 2} \Bigl(1+{1 \over \xi}\Bigr) \exp\biggl(- {y \over 4M\xi}\biggr), & \hspace{-1cm} \tfrac{1}{M} \leq x \leq M, \; \xi > 0
\end{align}
where $A_1(y) = A(y) \, (yM)^{1 \over 2} (1+y) (1+M) \exp({M \over 2}+{y \over 2})$; the first inequality follows by combining \eqref{eq:confhyp_prodformkern_bound} with the fact that $(1 + {y \over x} + {y \over \xi})^{-2a} \leq 1$ ($a \geq 0$), while the bounds for $x$ imply the second inequality. Clearly, \eqref{eq:confhyp_gentransl_cont_pf1} implies that $\int_0^\infty \! |q_a(x,y,\xi)|^q \ma(\xi) d\xi$ converges absolutely and uniformly in $x \in [{1 \over M}, M]$, and it follows that $\mathcal{T}_a^y f \in \mathrm{C}(0,\infty)$.

To prove the $L_p$-continuity of the translation, let $f \in \mathrm{C}_\mathrm{c}(0,\infty)$ and $1 < p <\infty$. Fix $M>0$ such that the support of $f$ is contained in $[{1 \over M},M]$. Interchanging the role of $x$ and $\xi$ in the estimate \eqref{eq:confhyp_gentransl_cont_pf1}, we easily see that
\begin{equation} \label{eq:confhyp_gentransl_cont_pf2}
|\mathcal{T}_a^{y+h\!} f(x)| \leq \|f\|_\infty \! \int_{1 \over M}^{M\!} q_a(x,y+h,\xi) \, \ma(\xi) d\xi \leq \|f\|_\infty \, A_2(y+h) \, x^{1 \over 2} \Bigl(1 + {1 \over x}\Bigr) \exp\biggl( -{y+h \over 4Mx} \biggr)
\end{equation}
where $A_2(y) = A_1(y) \int_{\mathsmaller{\mathsmaller{1/M}}}^M \ma(\xi) d\xi$. It is easy to check that the function $A_2(y)$ is locally bounded on $(0,\infty)$, so it follows from \eqref{eq:confhyp_gentransl_cont_pf2} that there exists $g \in L_p^a$ such that $|\mathcal{T}_a^{y+h} f(x)| \leq g(x)$ for all $0 < x < \infty$ and all $|h| < \delta$  (where $\delta > 0$ is sufficiently small). We have already proved that $(\mathcal{T}_a^y f)(x) \equiv (\mathcal{T}_a^x f)(y)$ is continuous in $y$, hence by $L_p$-dominated convergence we conclude that $\|\mathcal{T}_a^{y+h} f - \mathcal{T}_a^y f \|_{p,a} \to 0$ as $h \to 0$. As in the proof of the $L_p$-continuity of the ordinary translation, for general $f \in L_p^a$ the result is proved by taking a sequence of functions $f_n \in \mathrm{C}_\mathrm{c}(0,\infty)$ which tend to $f$ in $L_p^a$-norm. \\[-8pt]

\textbf{(d)} We start by studying the behavior as $x \to \infty$ of the integral $\int_{E_\delta}\! q_a(x,y,\xi)\, \ma(\xi) d\xi$, where $E_\delta = \{\xi \in (0,\infty): |y^{-1} - \xi^{-1}| > \delta\}$ and $\delta \in (0,y^{-1})$ is some fixed constant. We have
\begin{equation}
\label{eq:confhyp_gentransl_cont_pf3} q_a(x,y,\xi) e^{-\xi} \leq C\, {x\xi+xy+y\xi \over |x(\xi-y) +y\xi|} \exp\biggl( -{(x(\xi-y) +y\xi)^2 \over 8xy\xi} \biggr), \qquad x, \xi > 0
\end{equation}
(where $C < \infty$ is independent of $x$ and $\xi$). This follows by combining \eqref{eq:confhyp_prodformkern_bound} with the boundedness of the function $|t|e^{-t^2}$ and the inequality $(xy+x\xi+y\xi)^{-2a} \leq (x\xi)^{-2a}$. Furthermore, if $x \geq 2\delta^{-1}$ and $\xi \in E_\delta$, the inequalities 
\[
{x\xi+xy+y\xi \over |x(\xi-y) +y\xi|} = \biggl| 1+ {2\xi^{-1} \over y^{-1} - \xi^{-1} + x^{-1}} \biggr| \leq 1 + {4 \over \delta \xi}, \qquad \exp\biggl(- {x(\xi-y)^2 \over 8y\xi} - {y\xi \over 8x}\biggr) \leq \exp\biggl(- {(\xi-y)^2 \over 4\delta y\xi}\biggr)
\]
lead us to
\begin{equation}
\label{eq:confhyp_gentransl_cont_pf4} q_a(x,y,\xi) \xi^{-2a-1} e^{-\xi} \leq C \, \xi^{-2a-1} (1+\xi^{-1}) \exp\biggl( -{\xi \over 4} - {(\xi-y)^2 \over 4\delta y\xi} \biggr), \qquad x \geq \tfrac{2}{\delta},\; \xi \in E_\delta.
\end{equation} Since the right-hand side of \eqref{eq:confhyp_gentransl_cont_pf4} clearly belongs to $L_1(E_\delta)$, the dominated convergence theorem is applicable, and letting $x \to \infty$ in \eqref{eq:confhyp_gentransl_cont_pf3} we find that
\begin{equation} \label{eq:confhyp_gentransl_cont_pf5}
\lim_{x \to \infty} \int_{E_\delta}\! q_a(x,y,\xi)\, \ma(\xi) d\xi = \int_{E_\delta} \Bigl(\lim_{x \to \infty} q_a(x,y,\xi)\Bigr) \ma(\xi) d\xi = 0.
\end{equation}
Let us now fix $\eps > 0$, and write $V_\delta = \{\xi \in (0,\infty): |y^{-1} - \xi^{-1}| \leq \delta\}$. Since $f$ is continuous, we can choose $\delta > 0$ such that $|f(\xi) - f(y)| < \eps$ for all $\xi \in V_\delta$. By this choice of $\delta$ and the positivity of $q_a(x,y,\xi)$, we find
\begin{align*}
\bigl|(\mathcal{T}_a^y f)(x) - f(y)\bigr| & = \biggl|\int_0^\infty q_a(x,y,\xi) \bigl(f(\xi) - f(y)\bigr) \ma(\xi) d\xi\biggr| \\
& \leq \biggl|\int_{E_\delta}\! q_a(x,y,\xi) \bigl(f(\xi) - f(y)\bigr) \ma(\xi) d\xi \biggr| + \eps \int_{V_\delta} q_a(x,y,\xi) \ma(\xi) d\xi \\
& \leq 2\|f\|_\infty\! \int_{E_\delta}\! q_a(x,y,\xi)\, \ma(\xi) d\xi + \eps.
\end{align*}
By \eqref{eq:confhyp_gentransl_cont_pf5}, it follows that $\limsup_{x \to \infty} \bigl|(\mathcal{T}_a^y f)(x) - f(y)\bigr| \leq \eps$. Since $\eps$ is arbitrary, the proof of part (d) is finished. \\[-8pt]

\textbf{(e)} We begin by claiming that for each $\delta > 0$ we have $\int_\delta^\infty q_a(x,y,\xi) \ma(\xi) d\xi \to 0$ as $x \to 0$. Indeed, if $x < {\delta \over 2}$ and $\xi \geq \delta$, combining \eqref{eq:confhyp_gentransl_cont_pf3} with the inequalities
\[
{x\xi+xy+y\xi \over x(\xi-y) +y\xi} = 1+ {2\xi^{-1} \over y^{-1} + x^{-1} - \xi^{-1}} \leq 1+ {2\delta^{-1} \over y^{-1} + \delta^{-1}}, \qquad \exp\biggl(- {x(\xi-y)^2 \over 8y\xi} - {y\xi \over 8x}\biggr) \leq \exp\biggl(- {y\xi \over 4\delta}\biggr)
\]
we see that
\[ q_a(x,y,\xi) \xi^{-2a-1} e^{-\xi} \leq C\, \xi^{-2a-1} \exp\biggl( -{\xi \over 4} - {y\xi \over 4\delta} \biggr), \qquad x \leq {\delta \over 2},\; \xi \geq \delta
\]
where the right-hand side belongs to $L_1([\delta,\infty))$; hence, if we let $x \to 0$ in \eqref{eq:confhyp_gentransl_cont_pf3}, by dominated convergence we obtain
\begin{equation} \label{eq:confhyp_gentransl_cont_pf6}
\lim_{x \to 0} \int_\delta^\infty\! q_a(x,y,\xi)\, \ma(\xi) d\xi = \int_\delta^\infty\! \Bigl(\lim_{x \to 0} q_a(x,y,\xi)\Bigr) \ma(\xi) d\xi = 0.
\end{equation}
Let $f \in \mathrm{B}_\mathrm{b}(0,\infty)$ be such that $\lim_{x \to 0} f(x) = 0$, and let $\eps > 0$. Choose $M$ such that $|f(x)| < \eps$ for all $x \leq \delta$. Then,
\begin{align*}
|(\mathcal{T}_a^y f)(x)| & \leq \|f\|_\infty \int_\delta^\infty q_a(x,y,\xi)\, \ma(\xi) d\xi + \eps \int_0^\delta q_a(x,y,\xi)\, \ma(\xi) d\xi \\
& \leq \|f\|_\infty \int_\delta^\infty q_a(x,y,\xi)\, \ma(\xi) d\xi + \eps
\end{align*}
so that \eqref{eq:confhyp_gentransl_cont_pf6} yields $\limsup_{x \to 0} |(\mathcal{T}_a^y f)(x)| \leq \eps$, and hence $\lim_{x \to 0} (\mathcal{T}_a^y f)(x) = 0$ because $\eps$ is arbitrary.
\end{proof}

\subsection{Generalized convolution in the spaces $L_p^a$} \label{sec:Whit_convol_Lpconv}

The index Whittaker translation induces, in a standard way, a generalized convolution operator:

\begin{definition} \label{def:confhyp_conv_def}
Let $f,g: (0,\infty) \to \mathbb{C}$ be complex-valued functions and let $a \geq 0$. Write $\ma(\xi) = \xi^{-2a-1} e^{-\xi}$ and let $q_a(x,y,\xi)$ be defined as in \eqref{eq:confhyp_prodform_kerndef}. If the double integral
\[
(f \conv{a} g)(x) := \int_0^\infty (\mathcal{T}_a^x f)(\xi)\, g(\xi)\, \ma(\xi) d\xi = \int_0^\infty \! \int_0^\infty \! q_a(x,y,\xi)\, f(y)\, g(\xi)\, \ma(y) dy\, \ma(\xi) d\xi
\]
exists for almost every $0 < x < \infty$, then we call it the \emph{index Whittaker convolution (of order $a$)} of the functions $f$ and $g$.
\end{definition}

This convolution generalizes the convolution associated with the Kontorovich-Lebedev transform: indeed, in the case $a={1 \over 2}$ it is straightforward to verify, using \eqref{eq:prel_parcyl_expcase}, that
\[
(2\pi)^{-{1 \over 2}} x^{-{3 \over 2}} e^{-x} (f \conv{\!\mathsmaller{1/2}\!} g)(2x) = (F \conv{\!\mathsmaller{KL}\!} G)(x) 
\]
where $\conv{\!\!\mathsmaller{KL}\!\!}$ is the Kontorovich-Lebedev convolution operator (as defined e.g.\ in \cite[Section 4.1]{yakubovich1996}) and $F(x) = x^{-{3 \over 2}} e^{-x} f(2x)$, $G(x) = x^{-{3 \over 2}} e^{-x} g(2x)$.

From Definition \ref{def:confhyp_conv_def} it immediately follows that, for each $a \geq 0$, the index Whittaker convolution is positivity-preserving (i.e., $f \conv{a} g \geq 0$ whenever $f, g \geq 0$) and commutative (i.e., $f \conv{a} g = g \conv{a} f$). Moreover, the index Whittaker convolution satisfies Young's inequality:

\begin{proposition} \label{prop:confhyp_conv_Lrestimate}
Let $a \geq 0$ and $p_1,p_2 \in [1, \infty]$ such that ${1 \over p_1} + {1 \over p_2} \geq 1$. For $f \in L_{p_1}^a$ and $g \in L_{p_2}^a$, the convolution $f \conv{a} g$ is well-defined and, for $r \in [1, \infty]$ defined by ${1 \over r} = {1 \over p_1} + {1 \over p_2} - 1$, it satisfies
\[
\| f \conv{a} g \|_{r,a} \leq \| f \|_{p_1,a} \| g \|_{p_2,a}
\]
(in particular, $f \conv{a} g \in L_r^a$).
\end{proposition}

\begin{proof}
The proof relies on Proposition \ref{prop:confhyp_gentransl_props}(b) and the same reasoning as in the classical case; see e.g.\ the proof of Proposition 1.III.5 of \cite{trimeche1997}.
\end{proof}

Similar to the case of the classical convolution, the $\conv{a}$-convolution of two functions belonging to $L_p^a$ spaces with conjugate exponents defines a continuous function:

\begin{proposition}
Let $a \geq 0$ and $p, q \in [1,\infty]$ with ${1 \over p} + {1 \over q} = 1$. If $f \in L_p^a$ and $g \in L_q^a$, then $f \conv{a} g \in \mathrm{C}_\mathrm{b}(0,\infty)$.
\end{proposition}

\begin{proof}
The previous proposition ensures the boundedness of $f \conv{a} g$. For the continuity, let $x_0 > 0$; then for $1 < p < \infty$ we have
\begin{align*}
\bigl|(f \conv{a} g)(x) - (f \conv{a} g)(x_0)\bigr| & = \Bigl| \int_0^\infty \bigl((\mathcal{T}_a^x f)(\xi) - (\mathcal{T}_a^{x_0} f)(\xi)\bigr) g(\xi) \ma(\xi) d\xi \Bigr| \\
& \leq \| \mathcal{T}_a^{x\!} f - \mathcal{T}_a^{x_0\!} f \|_{p,a} \| g \|_{q,a} \to 0 \qquad \text{as } x \to x_0 
\end{align*}
by H\"{o}lder's inequality and Proposition \ref{prop:confhyp_gentransl_props}(c). In the case $p=\infty$ (and by symmetry $p=1$), the continuity of $f \conv{a} g$ follows by dominated convergence, using parts (a) and (c) of Proposition \ref{prop:confhyp_gentransl_props}.
\end{proof}

We now turn our attention to the connection between the $\conv{a}$-convolution and the index Whittaker transform, which will be our tool for establishing additional $L_p$ properties for the convolution. We shall consider the \emph{confluent hypergeometric form} of the index Whittaker transform, which we define by
\begin{equation} \label{eq:confhyp_transf}
(\varPsi_a f)(\tau) = \int_0^\infty f(x) x^{a+i\tau} \Psi(a+i\tau,1+2i\tau;x) \, \ma(x) dx, \qquad \tau \geq 0
\end{equation}
(later, complex values for $\tau$ shall also be considered).

Before stating the basic properties of this integral transform, we observe that, by transformation of the Whittaker differential equation, the function $x^{a+\nu} \Psi(a+\nu,1+2\nu;x)$ is a standard solution of the  confluent hypergeometric-type differential equation
\begin{equation} \label{eq:confhyp_Lop_ode}
\mathcal{L}_a w(x) = (\nu^2 - a^2) w(x),
\end{equation}
where $\mathcal{L}_a$ is the differential operator
\begin{equation} \label{eq:confhyp_Lop_def}
\mathcal{L}_a f = x^2 {d^2 f \over dx^2} - ((2a-1)x+x^2) {df \over dx}.
\end{equation}
The other standard solution of \eqref{eq:confhyp_Lop_ode} is $x^{a+\nu} {}_{1\!}F_1(a+\nu,1+2\nu;x)$, where ${}_{1\!}F_1(a+\nu,1+2\nu;x)$ denotes the confluent hypergeometric function of the first kind \cite[Chapter 13]{dlmf}.

\begin{theorem}
For $a > 0$, the index Whittaker transform \eqref{eq:confhyp_transf} defines an isometric isomorphism
\[
\varPsi_a: L_2^a \longrightarrow L_2\bigl((0,\infty); \rho_a(\tau) d\tau \bigr)
\]
where $\rho_a(\tau) := \pi^{-2} \tau \sinh(2\pi\tau)  \bigl| \Gamma\bigl(a + i\tau\bigr) \bigr|^2$, whose inverse is given by
\begin{equation} \label{eq:confhyp_inverse}
(\varPsi_a^{-1} \varphi)(x) = \int_0^\infty \! \varphi(\tau) \, x^{a+i\tau} \Psi(a+i\tau,1+2i\tau;x)\, \rho_a(\tau) d\tau
\end{equation}
the convergence of the integrals \eqref{eq:confhyp_transf} and \eqref{eq:confhyp_inverse} being understood with respect to the norm of the spaces $L_2\bigl((0,\infty); \rho_a(\tau) d\tau \bigr)$ and $L_2^a$ respectively. Moreover, the confluent hypergeometric-type differential operator \eqref{eq:confhyp_Lop_def} is connected with the Whittaker transform via the identity
\begin{equation} \label{eq:confhyp_transf_relLop}
\bigl[\varPsi_a (\mathcal{L}_a f)\bigr](\tau) = (\tau^2 + a^2) \ccdot \bigl(\varPsi_a f\bigr)(\tau), \qquad f \in \mathcal{D}_2^a
\end{equation}
where 
\[
\mathcal{D}_2^a := \Bigl\{f \in L_2^a \Bigm| f \text{ and } f'\! \text{ locally absolutely continuous on } (0,\infty), \; \mathcal{L}_a f \in L_2^a, \; \lim_{x \to \infty} x^{1-2a} e^{-x} f'(x) = 0 \Bigr\}.
\]
\end{theorem}

\begin{proof}
The index Whittaker transform in confluent hypergeometric form can be written as the composition $\varPsi_a f = \mathrm{W}_{\!{1 \over 2} - a} (\Theta_a f)$, where $\Theta_a: L_2^a \longrightarrow L_2\bigl((0,\infty); y^{-2} dy\bigr)$ is the isometric operator defined by
\[
(\Theta_a f)(x) := x^{{1 \over 2} - a} e^{-{x \over 2}} f(x), \qquad x > 0
\]
and $\mathrm{W}_{\!{1 \over 2} - a}$ is the operator \eqref{eq:intro_Whit_transf} of the index Whittaker transform in classical form. Therefore, the fact that $\varPsi_a$ is an isomorphism and the inversion formula follows from known results on the $L_2$-theory for the index Whittaker transform, cf.\ \cite[Section 3]{srivastava1998}. As for relation \eqref{eq:confhyp_transf_relLop}, it is due to the fact that the index Whittaker transform in confluent hypergeometric form arises from the spectral expansion of the differential operator \eqref{eq:confhyp_Lop_def}, cf.\ \cite[Section 3.2]{sousayakubovich2018} and \cite[Section 8]{weidmann1987}.
\end{proof}

We note that, for $a > 0$ and $\nu = i \tau$, the product formula \eqref{eq:confhyp_prodform} can be written as
\[
(xy)^{a+\nu} \Psi(a+\nu,1+2\nu;x) \Psi(a+\nu,1+2\nu;y) = [\varPsi_{a\,} q_a(x,y,\cdot)](\tau), \qquad (x, y > 0, \: a > 0, \: \tau \geq 0).
\]
Applying the inverse Whittaker transform \eqref{eq:confhyp_inverse}, we find that for $x, y, \xi > 0$ and $a > 0$ we have
\begin{equation} \label{eq:confhyp_prodformkern_intrep}
\begin{aligned}
& q_a(x,y,\xi) = \int_0^{\infty\!} (xy\xi)^{a+i\tau} \Psi(a+i\tau,1+2i\tau;x) \Psi(a+i\tau,1+2i\tau;y) \Psi(a+i\tau,1+2i\tau;\xi) \, \rho_a(\tau) d\tau 
\end{aligned}
\end{equation}
where the integral on the right-hand side converges absolutely, as can be verified using the asymptotic forms \eqref{eq:prel_Whit_asymtau} and $\bigl|\Gamma(a + i\tau)\bigr| \sim (2\pi)^{1 \over 2} \tau^{a - {1 \over 2}} \exp(-{\pi\tau \over 2})$, $\,\tau \to +\infty$ (cf.\ \cite[Equation 5.11.9]{dlmf}).

The following upper bound on the kernel of the index Whittaker transform turns out to be useful for studying the connection with the $\conv{a}$-convolution:

\begin{lemma} \label{lem:confhyp_leq1}
Let $a \geq 0$ and $\nu \in \mathbb{C}$ belonging to the strip $0 \leq \Re\nu \leq a$. Then,
\[
\bigl| x^{a+\nu} \Psi(a+\nu,1+2\nu;x) \bigr| \leq 1 \qquad \text{for all } x \geq 0.
\]
\end{lemma}

\begin{proof}
By relations 2.17.7.4 in \cite{prudnikovIII1990} and 3.7(6) in \cite{erdelyiI1953}, the confluent hypergeometric function admits the integral representation
\begin{align*}
& x^{a+\nu} \Psi(a+\nu,1+2\nu;x) = {x^{a+{1 \over 2}} \over 2} \int_1^\infty\! e^{-{x \over 2} (t-1)} \biggl({t-1 \over t+1}\biggr)^{\!{a \over 2}-{1 \over 4}} P_{-{1 \over 2} + \nu}^{{1 \over 2} - a}(t)\, dt \\
& \qquad\qquad = {2^{-a-{1 \over 2}} \pi^{-{1 \over 2}} \over \Gamma(a)} x^{a+{1 \over 2}}\! \int_1^\infty\! e^{-{x \over 2} (t-1)}\, (t-1)^{a-{1 \over 2}} \! \int_0^\pi \bigl(t + (t^2-1)^{1 \over 2} \cos s\bigr)^{-a+\nu} (\sin s)^{2a-1} ds\, dt\\[-2pt]
& & \hspace{-2.5cm} (\Re a > 0, \: \nu \in \mathbb{C})
\end{align*}
where $P_{-{1 \over 2} + \nu}^{{1 \over 2} - a}(t)$ is the associated Legendre function of the first kind \cite[Chapter III]{erdelyiI1953}. Consequently,
\begin{equation} \label{eq:confhyp_leqrealpart}
\begin{aligned}
\bigl| x^{a+\nu} \Psi(a+\nu,1+2\nu;x) \bigr| & \leq {2^{-\Re a-{1 \over 2}} \pi^{-{1 \over 2}} \over \Gamma(a)} x^{a+{1 \over 2}}\! \int_1^\infty\! e^{-{x \over 2} (t-1)}\, (t-1)^{\Re a-{1 \over 2}}\\[-1pt]
& \hspace{2cm} \times\! \int_0^\pi \bigl(t + (t^2-1)^{1 \over 2} \cos s\bigr)^{\Re\!(-a+\nu)} (\sin s)^{2\Re a-1} ds\, dt \\[2pt]
& = {\Gamma(\Re a) \over |\Gamma(a)|}\, x^{\Re\!(a+\nu)} \Psi\bigl(\Re\!(a+\nu),1+2\Re\nu;x\bigr) & \hspace{-1.5cm} (\Re a < \tfrac{1}{2}, \: \nu \in \mathbb{C}).
\end{aligned}
\end{equation}
In particular, if $a$ is a positive real number then $|x^{a+\nu} \Psi(a+\nu,1+2\nu;x)| \leq x^{a+\Re\nu} \Psi(a+\Re\nu,1+2\Re\nu;x)$; therefore, to conclude the proof we just need to prove that $|x^{a+\nu} \Psi(a+\nu,1+2\nu;x)| \leq 1$ for $a>0$ and $0 \leq \nu \leq a$ (the result for $a = 0$ being obtained by continuity). Indeed, \eqref{eq:prel_Whit_intrep} yields that
\begin{align*}
0 < x^{a+\nu} \Psi(a+\nu,1+2\nu;x) & = {1 \over \Gamma(a + \nu)} \int_0^\infty\! e^{-s} s^{a+\nu-1} \Bigl(1+{s \over x}\Bigr)^{\!-a+\nu} ds \\ 
& \leq {1 \over \Gamma(a + \nu)} \int_0^\infty\! e^{-s} s^{a + \nu-1} ds \\ 
& = 1 & \hspace{-2.5cm}(x > 0, \: a > 0, \: 0 \leq \nu \leq a)
\end{align*}
where the last inequality holds because $-a+\nu \leq 0$.
\end{proof}

The fundamental properties which connect the index Whittaker translation and convolution with the index Whittaker transform and its associated differential operator \eqref{eq:confhyp_Lop_def} are given in the following proposition.

\begin{proposition} \label{prop:confhyp_conv_indextransf}
Let $y > 0$, $a > 0$ and $\tau \geq 0$. Then: \\[-8pt]

\textbf{(a)} If $f \in L_2^a$, then $\:\bigl(\varPsi_a(\mathcal{T}_a^y f)\bigr)(\tau) = y^{a+i\tau} \Psi(a+i\tau,1+2i\tau;y) \, (\varPsi_a f)(\tau)$; \\[-8pt]

\textbf{(b)} If $f \in L_2^a$ and $g \in L_1^a$, then $\:\bigl(\varPsi_a(f \conv{a} g) \bigr)(\tau) = (\varPsi_a f)(\tau)\, (\varPsi_a g)(\tau)$; \\[-8pt]

\textbf{(c)} If $f \in L_2^a$ and $g \in L_1^a$, then $\:\mathcal{T}_a^y(f \conv{a} g) = (\mathcal{T}_a^y f) \conv{a} g$; \\[-8pt]

\textbf{(d)} If $f \in \mathcal{D}_2^a$, then $\:\mathcal{T}_a^y f \in \mathcal{D}_2^a$ and $\:\mathcal{L}_a (\mathcal{T}_a^y f) = \mathcal{T}_a^y (\mathcal{L}_a f)$; \\[-8pt]

\textbf{(e)} If $f \in \mathcal{D}_2^a$ and $g \in L_1^a$, then $\,f \conv{a} g \in \mathcal{D}_2^a$ and $\:\mathcal{L}_a (f \conv{a} g) = (\mathcal{L}_a f) \conv{a} g$.
\end{proposition}

\begin{proof}
\textbf{(a)} Let $f \in L_1^a \cap L_2^a$. The Macdonald-type formula \eqref{eq:confhyp_prodform}, combined with Fubini's theorem, gives
\begin{align*}
\bigl(\varPsi_a(\mathcal{T}_a^y f)\bigr)(\tau) & = \int_0^\infty\! \int_0^\infty \! f(\xi) q_a(x,y,\xi)\, \ma(\xi) d\xi\, x^{a+i\tau} \Psi(a+i\tau,1+2i\tau;x)\, \ma(x) dx \\[-0.5pt]
& =  y^{a+i\tau} \Psi(a+i\tau,1+2i\tau;y) \int_0^\infty \! f(\xi) \xi^{a+i\tau} \Psi(a+i\tau,1+2i\tau;\xi)\, \ma(\xi) d\xi \\
& = y^{a+i\tau} \Psi(a+i\tau,1+2i\tau;y)\, (\varPsi_a f)(\tau).
\end{align*}
By denseness and continuity, the equality extends to all $f \in L_2^a$, as required. \\[-8pt]

\textbf{(b)} For $f \in L_1^a \cap L_2^a$ and $g \in L_1^a$ we have
\begin{equation} \label{eq:confhyp_conv_indextransf_pf} \begin{aligned}
\bigl(\varPsi_a(f \conv{a} g) \bigr)(\tau) & = \int_0^\infty\! \int_0^\infty\! (\mathcal{T}_a^x f)(\xi)\, g(\xi)\, \ma(\xi) d\xi\,  x^{a+i\tau} \Psi(a+i\tau,1+2i\tau;x)\, \ma(x) dx \\
& = \int_0^\infty\! g(\xi) \bigl(\varPsi_a(\mathcal{T}_a^\xi f)\bigr)\!(\tau)\, \ma(\xi) d\xi \\
& = (\varPsi_a f)(\tau) \int_0^\infty\!g(\xi)\,  \xi^{a+i\tau} \Psi(a+i\tau,1+2i\tau;\xi)\, \ma(\xi) d\xi = (\varPsi_a f)(\tau)\, (\varPsi_a g)(\tau)
\end{aligned}
\end{equation}
where we have used Fubini's theorem and part (a). Again, denseness yields the result. \\[-8pt]

\textbf{(c)} By the previous properties, 
\[
\varPsi_a\bigl[\mathcal{T}_a^y(f \conv{a} g)\bigr](\tau) = \varPsi_a\bigl[(\mathcal{T}_a^y f) \conv{a} g\bigr](\tau) = y^{a+i\tau} \Psi(a+i\tau,1+2i\tau;y) \, (\varPsi_a f)(\tau) \,
(\varPsi_a g)(\tau).
\]
Since both $\mathcal{T}_a^y(f \conv{a} g)$ and $(\mathcal{T}_a^y f) \conv{a} g$ are elements of the space $L_2^a$ (see Proposition \ref{prop:confhyp_conv_L2estimate} below), this implies that $\mathcal{T}_a^y(f \conv{a} g) = (\mathcal{T}_a^y f) \conv{a} g$. \\[-8pt]

\textbf{(d)} Since the integral transform $\varPsi_a$ is generated by the spectral expansion of the operator \eqref{eq:confhyp_Lop_def}, it is known from the general spectral theory of linear differential operators  \cite{weidmann1987} that a function $f \in L_2^a$ belongs to $\mathcal{D}_2^a$ if and only if $(\tau^2 + a^2) \ccdot (\varPsi_a f)(\tau) \in L_2\bigl((0,\infty); \rho_a(\tau) d\tau \bigr)$. Using this fact and the inequality $|x^{a+i\tau} \Psi(a+i\tau,1+2i\tau;x)| \leq 1$ (Lemma \ref{lem:confhyp_leq1}), it is easy to see that $\mathcal{T}_a^y f \in \mathcal{D}_2^a$. The identity $\mathcal{L}_a (\mathcal{T}_a^y f) = \mathcal{T}_a^y (\mathcal{L}_a f)$ holds because the Whittaker transform of both sides equals $(\tau^2 + a^2) y^{a+i\tau} \Psi(a+i\tau,1+2i\tau;y) (\varPsi_a f)(\tau)$. \\[-8pt]

\textbf{(e)} The proof is similar to that of (d).
\end{proof}

We have seen in Proposition \ref{prop:confhyp_conv_Lrestimate} that if $f \in L_2^a$ and $g \in L_p^a$ ($1 \leq p < 2$) then the Whittaker convolution $f \conv{a} g$ exists and belongs to $L_{\!2p \over 2-p}^a$. Using the index Whittaker transform, this result can be strengthened as follows:

\begin{proposition} \label{prop:confhyp_conv_L2estimate}
Let $f \in L_2^a$ and $g \in L_p^a$, where $1 \leq p < 2$ and $a > 0$. Then $f \conv{a} g \in L_2^a$, and we have
\[
\|f \conv{a} g\|_{L_2^a} \leq C_p \|f\|_{L_2^a} \|g\|_{L_p^a}
\]
where $C_p = \bigl\| x^{a} \Psi(a,1;x) \bigr\|_{L_q^a} < \infty$ (being ${1 \over p} + {1 \over q} = 1$).
\end{proposition}

\begin{proof}
The fact that $\|x^{a} \Psi(a,1;x)\|_{L_q^a}$ is finite for each $2 < q \leq \infty$ is easily verified using the limiting forms \eqref{eq:prel_Whit_asyminfty}, \eqref{eq:prel_Whit_asymzero}. Now, for $f, g \in \mathrm{C}_\mathrm{c}(0,\infty)$ we have
\[
\| f \conv{a} g\|_{L_2^a} = \bigl\| (\varPsi_a f) \ccdot (\varPsi_a g) \bigr\|_{L_2(\rho_a)} \leq \, \smash{\sup_{\tau \geq 0}} \bigl|(\varPsi_a g)(\tau)\bigr| \, \ccdot \bigl\| \varPsi_a f \bigr\|_{L_2(\rho_a)} \leq \bigl\| x^{a} \Psi(a,1;x) \bigr\|_{L_q^a} \|g\|_{L_p^a} \|f\|_{L_2^a}
\]
where we denoted $L_2(\rho_a) = L_2\bigl((0,\infty); \rho_a(\tau) d\tau\bigr)$; we have used the isometric property of the index Whittaker transform, and the final step relies on the inequality $|x^{a+i\tau} \Psi(a+i\tau,1+2i\tau;x)| \leq x^{a} \Psi(a,1;x)$ (proved in \eqref{eq:confhyp_leqrealpart}) and on H\"{o}lder's inequality. As usual, the result for $f \in L_2^a$ and $g \in L_p^a$ follows from the denseness of $\mathrm{C}_\mathrm{c}(0,\infty)$ in these $L_p$ spaces.
\end{proof}

\begin{corollary}
\textbf{(a)} If $f, g \in L_2^a$ ($a > 0$), then $f \conv{a} g \in L_q^a$ for all $2 < q \leq \infty$, with
\[
\| f \conv{a} g \|_{L_q^a} \leq C_q \|f\|_{L_2^a} \|g\|_{L_2^a}
\]
being $C_q = \bigl\| x^{a} \Psi(a,1;x) \bigr\|_{L_q^a}$. \\[-8pt]

\textbf{(b)} Let $1 \leq p_1 < 2$ and $1 \leq p_2 \leq 2$ such that ${1 \over p_1} + {1 \over p_2} \leq {3 \over 2}$. Let $r$ be defined by ${1 \over r} = {1 \over p_1} + {1 \over p_2} - 1$. If $f \in L_{p_1}^a$ and $g \in L_{p_2}^a$ ($a > 0$), then $f \conv{a} g \in L_s^a$ for all $s \in [2,r]$.
\end{corollary}

\begin{proof}
Both results can be deduced from Proposition \ref{prop:confhyp_conv_L2estimate} by arguing as in the proofs of Theorem 5.5(ii) and Corollary  5.6 of \cite{flenstedjensen1973}. 
\end{proof}

\subsection{The convolution Banach algebra $L^{a,\nu}$} \label{sec:Whit_convol_banachalg}

In this subsection we focus on the properties of the index Whittaker convolution in the family of spaces $\{L^{a,\nu}\}_{\nu \geq 0}$, where
\[
L^{a,\nu\!} := L_1\bigl((0,\infty); x^{a+\nu} \Psi(a+\nu,1+2\nu;x)\, \ma(x) dx\bigr) \qquad\;\; (0 < a < \infty, \: \nu \geq 0)
\]
being $\ma(x) = x^{-2a-1} e^{-x}$. We observe that, by the limiting forms of the confluent hypergeometric function of the second kind,
\begin{equation} \label{eq:confhyp_L1nu_equiv}
\begin{aligned}
f \in L^{a,\nu} & \qquad \text{if and only if} \qquad f \in L_1\bigl((0,1]; x^{-a-\nu-1} dx \bigr) \cap L_1\bigl([1,\infty); x^{-2a-1} e^{-x} dx\bigr) & \hspace{-3mm} (\nu > 0) \\
f \in L^{a,0} & \qquad \text{if and only if} \qquad f \in L_1\bigl((0,1]; -x^{-a-1} \log x \, dx \bigr) \cap L_1\bigl([1,\infty); x^{-2a-1} e^{-x} dx \bigr)
\end{aligned}
\end{equation}
and therefore the spaces $L^{a,\nu}$ are ordered:
\begin{equation} \label{eq:confhyp_L1nu_orderrel}
L^{a,\nu_1} \subset L^{a,\nu_2} \; \text{ whenever } \; \nu_1 > \nu_2.
\end{equation}
It is also interesting to note that the family $\{L^{a,\nu}\}_{\nu \geq 0}$ contains the space $L_1^a$. Indeed, \eqref{eq:prel_Whit_elementary} yields $\Psi(2a,1+2a;x) \equiv x^{-2a}$, which means that $L_1^a = L^{a,a}$.

The following lemma gives some properties of the index Whittaker transform in the spaces $L^{a,\nu}$ which will be needed in what follows.

\begin{lemma} \label{lem:confhyp_transf_L1nu}
Let $0 < a < \infty$ and $\nu \geq 0$. If $f \in L^{a,\nu}$, then its index Whittaker transform $\varPsi_a f$ is well-defined as an absolutely convergent integral, and it satisfies
\begin{equation} \label{eq:confhyp_transf_taulim}
(\varPsi_a f)(\tau) \xrightarrow[\,\tau \to \infty\,]{} 0.
\end{equation}
Moreover, if $\varPsi_a f$ is identically zero, then $f(x) = 0$ for almost every $x>0$.
\end{lemma}

\begin{proof}
The absolute convergence of the integral \eqref{eq:confhyp_transf} follows from the inequality \eqref{eq:confhyp_leqrealpart} and the inclusion $L^{a,\nu} \subset L^{a,0}$. By \eqref{eq:prel_Whit_asymtau} we have $\lim_{\tau \to \infty} x^{a+i\tau} \Psi(a+i\tau,1+2i\tau;x) = 0$ for each $x>0$, hence dominated convergence gives \eqref{eq:confhyp_transf_taulim}.

Now, suppose that $(\varPsi_a f)(\tau) = 0$ for all $\tau \geq 0$. Since, by integral 2.16.8.3 in \cite{prudnikovII1986}, the Whittaker transform can be written as
\begin{align*}
(\varPsi_a f)(\tau) & = {2^{2-2a} \over |\Gamma(a + i\tau)|^2} \int_0^\infty f(x) \! \int_0^\infty \exp\bigl(-\tfrac{t^2}{4x}\bigr) t^{2a-1} K_{2i\tau}(t) dt\, \ma(x) dx \\
& = {2^{2-2a} \over |\Gamma(a + i\tau)|^2} \int_0^\infty K_{2i\tau}(t) \, t^{2a-1\!} \! \int_0^\infty \exp\bigl(-\tfrac{t^2}{4x}\bigr) f(x) \ma(x) dx\, dt
\end{align*}
(the application of Fubini's theorem being easily justified), this means that the Kontorovich-Lebedev transform of the function $t^{2a-1} \! \int_0^\infty \exp\bigl(-\tfrac{t^2}{4x}\bigr) f(x) \, \ma(x) dx$ vanishes identically. By the injectivity property of the Kontorovich-Lebedev transform (see \cite[Theorem 6.5]{yakubovichluchko1994}), it follows that 
\[
\int_0^\infty \exp\bigl(-\tfrac{t^2}{4x}\bigr) f(x) \, \ma(x) dx = 0 \quad\; \text{for almost every } t>0.
\]
Now, the left-hand side is the Laplace transform of the function $x^{2a-1} e^{-1/x} f({1 \over x})$ evaluated at $z = {t^2 \over 4}$. Using the inverse theorem for the Laplace transform and a reasoning similar to that in the proof of Theorem 6.5 of \cite{yakubovichluchko1994}, it follows that $f(x) = 0$ for almost every $x>0$, and this completes the proof.
\end{proof}

\begin{proposition} \label{prop:confhyp_conv_L1nu_norm}
Let $a > 0$ and $\nu \geq 0$. For $f,g \in L^{a,\nu}$, the index Whittaker convolution $f \conv{a} g$ is well-defined and satisfies
\[
\|f \conv{a} g\|_{L^{a,\nu}} \leq \|f\|_{L^{a,\nu}} \|g\|_{L^{a,\nu}}
\]
(in particular, $f \conv{a} g \in L^{a,\nu}$). Moreover, properties (a) and (b) in Proposition \ref{prop:confhyp_conv_indextransf} are valid when $f$ and $g$ belong to $L^{a,\nu}$ and $\tau$ is a complex number such that $|\Im \tau| \leq \nu$.
\end{proposition}

\begin{proof}
For $a$ and $\nu$ as in the statement, we have
\begin{align*}
\|f \conv{a} g\|_{L^{a,\nu}} & \leq \int_0^\infty \! \int_0^\infty \! \int_0^\infty \! |f(y)| q_a(x,y,\xi) \ma(y) dy\, |g(\xi)| \ma(\xi) d\xi \, x^{a+\nu} \Psi(a+\nu,1+2\nu;x) \ma(x) dx \\
& = \int_0^\infty \! \int_0^\infty \! \int_0^\infty \!  x^{a+\nu} \Psi(a+\nu,1+2\nu;x) q_a(x,y,\xi) \ma(x) dx\, |f(y)| \ma(y) dy \, |g(\xi)| \ma(\xi) d\xi \\
& = \int_0^\infty \! |f(y)|\, y^{a+\nu} \Psi(a+\nu,1+2\nu;y)\ma(y) dy \int_0^\infty \!  |g(\xi)|\, \xi^{a+\nu} \Psi(a+\nu,1+2\nu;\xi) \ma(\xi) d\xi \\
& = \|f\|_{L^{a,\nu}} \|g\|_{L^{a,\nu}}
\end{align*}
where the positivity of the integrand justifies the change of order of integration, and the second equality follows from the product formula \eqref{eq:confhyp_prodform}. The final statement is proved using the same calculations as before.
\end{proof}

\begin{corollary} \label{cor:confhyp_conv_L1nu_banalg}
The Banach space $L^{a,\nu}$, equipped with the convolution multiplication $f \cdot g \equiv f \conv{a} g$, is a commutative Banach algebra without identity element.
\end{corollary}

\begin{proof}
Proposition \ref{prop:confhyp_conv_L1nu_norm} shows that the Whittaker convolution defines a binary operation on $L^{a,\nu}$ for which the norm is submultiplicative. The commutativity and associativity of the Whittaker convolution in the space $L^{a,\nu}$ follows from the property $\bigl(\varPsi_a(f \conv{a} g) \bigr)(\tau) = (\varPsi_a f)(\tau)\, (\varPsi_a g)(\tau)$ and the injectivity property of Lemma \ref{lem:confhyp_transf_L1nu}.

Suppose now that there exists $\mathrm{e} \in L^{a,\nu}$ such that $f \conv{a} \mathrm{e} = f$ for all $f \in L^{a,\nu}$. This means that
\[
(\varPsi_a f)(\tau) \, (\varPsi_a \mathrm{e})(\tau) = (\varPsi_a f)(\tau) \qquad \text{for all } f \in L^{a,\nu} \text{ and } \tau \geq 0.
\]
Clearly, this implies that $(\varPsi_a \mathrm{e})(\tau) = 1$ for all $\tau \geq 0$, which contradicts Lemma \ref{lem:confhyp_transf_L1nu}. This shows that there exists no identity element for the Whittaker convolution on the space $L^{a,\nu}$.
\end{proof}

In order to prove the Wiener-Lévy theorem for the index Whittaker transform, we need the following lemma on the existence of additional solutions for the functional equation associated with the product formula \eqref{eq:confhyp_prodform}.

\begin{lemma} \label{lem:confhyp_prodform_functeq}
Let $a > 0$ and $\nu \geq 0$. Suppose that the function $\omega(x)$ is such that there exists $C>0$ for which
\begin{equation} \label{eq:confhyp_prodform_functeq_essbound}
\bigl|\omega(x)\bigr| \leq C\,  x^{a+\nu} \Psi(a+\nu,1+2\nu;x) \qquad \text{for almost every } x > 0
\end{equation}
and that $\omega(x)$ is a nontrivial solution of the functional equation
\begin{equation} \label{eq:confhyp_prodform_functeq}
\omega(x) \omega(y) = \int_0^\infty \omega(\xi)\, q_a(x,y,\xi)\, \ma(\xi) d\xi \qquad (x,y > 0).
\end{equation}
Then $\omega(x) = x^{a+\rho} \Psi(a+\rho,1+2\rho;x)$ for some $\rho \in \mathbb{C}$ with $|\mathrm{Re}\, \rho| \leq \nu$. 
\end{lemma}

\begin{proof}
To start, we claim that
\begin{align} \label{eq:confhyp_prodform_functeq_pf1}
& \mathcal{L}_{a,x} \, q_a(x,y,\xi) = \mathcal{L}_{a,y} \, q_a(x,y,\xi) = \\
\nonumber & \qquad = - \! \int_0^{\infty\!} (xy\xi)^{a+i\tau} \Psi(a+i\tau,1+2i\tau;x) \Psi(a+i\tau,1+2i\tau;y) \Psi(a+i\tau,1+2i\tau;\xi) \, (\tau^2 + a^2) \rho_a(\tau) d\tau
\end{align}
where $\mathcal{L}_{a,x}$ and $\mathcal{L}_{a,y}$ denote the differential operator \eqref{eq:confhyp_Lop_def} acting on the variables $x$ and $y$ respectively. Indeed, this identity is obtained via differentiation of \eqref{eq:confhyp_prodformkern_intrep} under the integral sign; recall that $x^{a+\nu} \Psi(a+\nu,1+2\nu;x)$ satisfies the differential equation \eqref{eq:confhyp_Lop_ode}. The validity of differentiation under the integral sign is justified by the absolute and locally uniform convergence of the  differentiated integrals, which is verified in a straightforward way using the differentiation formula \eqref{eq:prel_confhyp_differentiate2} and the asymptotic expansion \eqref{eq:prel_Whit_asymtau}.

Now, assuming that the right-hand side of the functional equation \eqref{eq:confhyp_prodform_functeq} can also be differentiated under the integral sign, it follows from \eqref{eq:confhyp_prodform_functeq_pf1} that 
\begin{equation} \label{eq:confhyp_prodform_functeq_pf2}
\bigl( \mathcal{L}_{a,x\,} \omega(x) \bigr)\, \omega(y) = \bigl( \mathcal{L}_{a,y\,} \omega(y) \bigr)\, \omega(x) \qquad\;\; (x, y > 0).
\end{equation}
Here the possibility of interchanging derivative and integral follows again from the locally uniform convergence of the differentiated integrals, which can be straightforwardly checked using \eqref{eq:confhyp_prodform_functeq_essbound}, the identity
\[
{\partial \over \partial x} q_a(x,y,\xi) = \Bigl( 1 + {a \over x} \Bigr) q_a(x,y,\xi) + \Bigl( {y\xi \over x} - y - \xi \Bigr)\, q_{a - {1 \over 2}}(x,y,\xi)
\]
(which is a consequence of \eqref{eq:prel_parcyl_differentiate}) and the upper bound \eqref{eq:confhyp_prodformkern_bound} for the function $q_a(x,y,\xi)$.

Notice that \eqref{eq:confhyp_prodform_functeq_pf2} holds for arbitrary values of $x$ and $y$. Therefore, we must have
\[
{\mathcal{L}_{a,x\,} \omega(x) \over \omega(x)} = {\mathcal{L}_{a,y\,} \omega(y) \over \omega(y)} = \lambda
\]
for some $\lambda \in \mathbb{C}$, meaning that $\omega(x)$ is a solution of the confluent hypergeometric-type equation
\[
\mathcal{L}_a \omega(x) = (\rho^2 - a^{2\,}) \omega(x)
\]
where $\rho$ is a complex number such that $\rho^2 = \lambda + a^2$. This implies that $\omega(x)$ is a linear combination of $x^{a+\rho} \Psi(a+\rho,1+2\rho;x)$ and $x^{a+\rho} {}_{1\!}F_1(a+\rho,1+2\rho;x)$. But ${}_{1\!}F_1(a+\rho,1+2\rho;x)$ is, for all $\rho \in \mathbb{C}$, unbounded as $x$ goes to infinity \cite[Equation 13.2.23]{dlmf}, violating \eqref{eq:confhyp_prodform_functeq_essbound}. In addition, the limiting forms for $\Psi(a+\rho,1+2\rho;x)$ (which follow from \eqref{eq:prel_Whit_asyminfty}, \eqref{eq:prel_Whit_asymzero}) show that $\bigl| x^{a+\rho} \Psi(a+\rho,1+2\rho;x) \bigl| \leq C\, x^{a+\nu} \Psi(a+\nu,1+2\nu;x)$ holds if and only if $|\mathrm{Re}\, \rho| \leq \nu$. Therefore, we must have $\omega(x) = x^{a+\rho} \Psi(a+\rho,1+2\rho;x)$ for $\rho$ belonging to the strip $|\mathrm{Re}\, \rho| \leq \nu$.
\end{proof}

As a consequence of Corollary \ref{cor:confhyp_conv_L1nu_banalg}, Lemma \ref{lem:confhyp_prodform_functeq} and the elementary theory of Banach algebras, we find that an analogue of the Wiener-Lévy theorem is valid for the index Whittaker transform:

\begin{theorem} \label{thm:confhyp_transf_wiener}
Let $f \in L^{a,\nu}$, $\lambda \in \mathbb{C}$ and $\nu > 0$. Suppose that $\lambda + (\varPsi_a f)(\tau) \neq 0$ for all $\tau$ in the closed strip $|\Im \tau| \leq \nu$, including infinity. Then there exists a unique function $\eta \in L^{a,\nu}$ such that
\begin{equation} \label{eq:confhyp_transf_wiener}
{1 \over \lambda + (\varPsi_a f)(\tau)} = \lambda + (\varPsi_a \eta)(\tau) \qquad (|\Im \tau| \leq \nu).
\end{equation}
\end{theorem}

\begin{proof}
The proof is entirely analogous to the proof of Theorem 15.15 of \cite{yakubovichluchko1994}, appealing to Corollary \ref{cor:confhyp_conv_L1nu_banalg} and Lemma \ref{lem:confhyp_prodform_functeq} in place of the analogous results for the Kontorovich-Lebedev convolution.
\end{proof}

\begin{remark}
The converse of this theorem is also true, i.e.\ if for some $\tau_0$ with $|\Im \tau_0| \leq \nu$ we have $\lambda + (\varPsi_a f)(\tau_0) = 0$, then no function $\eta \in L^{a,\nu}$ can satisfy \eqref{eq:confhyp_transf_wiener}. Indeed, from \eqref{eq:confhyp_L1nu_orderrel} and the condition $\eta \in L^{a,\nu}$ it follows that the integral defining $(\varPsi_a \eta)(\tau)$ converges absolutely whenever $|\Im \tau| \leq \nu$, so if $\lambda + (\varPsi_a f)(\tau_0) = 0$ then \eqref{eq:confhyp_transf_wiener} will fail at $\tau = \tau_0$, regardless of the choice of $\eta \in L^{a,\nu}$.
\end{remark}

\section{Application to convolution integral equations} \label{chap:Whit_integraleq}

In this final section we demonstrate that the index Whittaker convolution, and especially the analogue of the Wiener-Lévy theorem proved above, can be used to study the existence of solution for integral equations of the second kind which can be represented as index Whittaker convolution equations, in the sense defined as follows:

\begin{definition}
The integral equation of the second kind
\begin{equation} \label{eq:confhyp_2ndkindinteq}
f(x) + \int_0^\infty J(x,y) f(y) \,  dy = h(x),
\end{equation}
where $h$ is a known function and $f$ is to be determined, is said to be a \textit{index Whittaker convolution equation} if there exists $a > 0$ and $\theta \in L^{a,0}$ such that $J(x,y) = (\mathcal{T}_a^x \theta)(y) \, \ma(y) \equiv (\mathcal{T}_a^x \theta)(y) \, y^{-2a-1} e^{-y}$. In other words, \eqref{eq:confhyp_2ndkindinteq} is a index Whittaker convolution equation if it can be written in the form 
\begin{equation} \label{eq:confhyp_convinteq}
f(x) + (f \conv{a} \theta)(x) = h(x)
\end{equation}
for some $a > 0$ and $\theta \in L^{a,0}$.
\end{definition}

Suppose that $h, \theta \in L^{a,\nu}$ (being $a > 0$ and $\nu \geq 0$), and consider the $\conv{a}$-convolution integral equation \eqref{eq:confhyp_convinteq}. Applying the index Whittaker transform to both sides of the convolution equation, we get
\begin{equation} \label{eq:confhyp_convinteq_transf}
(\varPsi_a f)(\tau) \bigl[ 1 + (\varPsi_a \theta)(\tau) \bigr] = (\varPsi_a h)(\tau) \qquad (|\Im \tau| \leq \nu).
\end{equation}
Now, Theorem \ref{thm:confhyp_transf_wiener} and the subsequent remark show that the condition
\[
1 + (\varPsi_a \theta)(\tau) \neq 0 \quad \text{throughout the strip } |\Im \tau| \leq \nu
\]
is a necessary and sufficient condition for the existence of a unique $\eta \in L^{a,\nu}$ satisfying
\begin{equation} \label{eq:confhyp_convinteq_wiener}
{1 \over 1 + (\varPsi_a \theta)(\tau)} = 1 + (\varPsi_a \eta)(\tau) \qquad (|\Im \tau| \leq \nu),
\end{equation}
and if this holds then from \eqref{eq:confhyp_convinteq_transf} we obtain $(\varPsi_a f)(\tau) = (\varPsi_a h)(\tau) \bigl[ 1 + (\varPsi_a \eta)(\tau) \bigr] \,$ ($|\Im \tau| \leq \nu$) or, equivalently,
\begin{equation} \label{eq:confhyp_convinteq_sol}
f(x) = h(x) + (h \conv{a} \eta)(x) = h(x) + \int_0^\infty \! J_\eta(x,y) \, h(y) \, dy
\end{equation}
where $J_\eta(x,y) = (\mathcal{T}_a^x \eta)(y) \, \ma(y)$. In summary, we have proved the following:

\begin{theorem} \label{thm:confhyp_convinteq}
Let $J(x,y) = (\mathcal{T}_a^x \theta)(y) \, \ma(y)$ where $\theta \in L^{a,\nu}$ ($a > 0$, $\nu \geq 0$), and suppose that $1 + (\varPsi_a \theta)(\tau) \neq 0$ for all $\tau$ in the strip $|\Im \tau| \leq \nu$, including infinity. Then the integral equation \eqref{eq:confhyp_convinteq} has, for any $h \in L^{a,\nu}$ a unique solution $f \in L^{a, \nu}$ which can be represented in the form \eqref{eq:confhyp_convinteq_sol} for some $\eta \in L^{a,\nu}$. Conversely, if $1 + (\varPsi_a \theta)(\tau_0) = 0$ for some $\tau_0$ with $|\Im \tau_0| \leq \nu$, then the equation \eqref{eq:confhyp_convinteq} is not solvable in the space $L^{a,\nu}$.
\end{theorem}

We point out that as long as ${(\varPsi_a \theta)(\tau) \over 1 + (\varPsi_a \theta)(\tau)} = O(\tau^{-2})$, the representation \eqref{eq:confhyp_convinteq_sol} for the solution of the integral equation can be rewritten as
\begin{equation} \label{eq:confhyp_convinteq_sol_closform}
f(x) = h(x) - \int_0^\infty \! \int_0^\infty \! {(\varPsi_a \theta)(\tau) \over 1 + (\varPsi_a \theta)(\tau)} (xy)^{a+i\tau} \Psi(a+i\tau,1+2i\tau;x) \Psi(a+i\tau,1+2i\tau;y)\, \rho_a(\tau) d\tau \, h(y) \, \ma(y) dy
\end{equation}
(here we used \eqref{eq:confhyp_inverse} and Proposition \ref{prop:confhyp_conv_indextransf}(a)). In many cases of interest, the index Whittaker transform $(\varPsi_a \theta)(\tau)$ can be computed in closed form using integration formulas for the confluent hypergeometric function (see \cite[Section 2.19]{prudnikovIII1990}), so that \eqref{eq:confhyp_convinteq_sol_closform} becomes an explicit expression for the solution of the convolution integral equation, which can be evaluated using numerical integration.

The index Whittaker translation of the power function $\theta(x) = x^\beta$, whose closed form was computed in Lemma \ref{lem:confhyp_gentransl_power}, yields a large family of $\conv{a}$-convolution integral equations to which this theorem can be applied:

\begin{corollary} \label{cor:confhyp_convinteq1}
Let $h \in L^{a,\nu}$ ($a>0$, $\nu \geq 0$), $\lambda \in \mathbb{C}$, and $\beta \in \mathbb{C}$ with $\Re\beta > a+\nu$. The integral equation
\begin{equation} \label{eq:confhyp_convinteq_cor1}
f(x) + \lambda \int_0^\infty (xy)^\beta \Psi(\beta, 1-2a+2\beta; x+y) \, f(y) \, \ma(y) dy = h(x),
\end{equation}
has a unique solution $f \in L^{a,\nu}$ if and only if the condition
\[ \label{eq:tilW_conv_integraleq_solvcond}
\Gamma(\beta) + \lambda\, 2^{\beta-2a}\, \Gamma\bigl(\beta - a + i\tau\bigr) \Gamma\bigl(\beta - a - i\tau\bigr) \neq 0
\]
holds for all $\tau \in \mathbb{C}$ in the strip $|\Im \tau| \leq \nu$, including infinity.
\end{corollary}

\begin{proof}
Let $\theta(x) = \lambda\, x^\beta$. It is clear from \eqref{eq:confhyp_L1nu_equiv} that $\theta \in L^{a,\nu}$. We have seen in Lemma \ref{lem:confhyp_gentransl_power} that 
\[
(\mathcal{T}_a^x \theta)(y) = \lambda \, (xy)^\beta \Psi(\beta, 1-2a+2\beta; x+y).
\]
The index Whittaker transform $\varPsi_a \theta$ is computed using relation 2.19.3.7 in \cite{prudnikovIII1990}:
\[
(\varPsi_a \theta)(\tau) = \lambda \! \int_0^\infty \! x^{\beta+a+i\tau} \Psi(a+i\tau,1+2i\tau;x) \, \ma(x) dx = {\lambda \over \Gamma(\beta)}\, \Gamma(\beta - a + i\tau) \Gamma(\beta - a - i\tau), \quad\;\; |\Im \tau| \leq \nu.
\]
The corollary is therefore obtained by setting $\theta(x) = \lambda x^\beta$ in Theorem \ref{thm:confhyp_convinteq}.
\end{proof}

It should be emphasized that Theorem \ref{thm:confhyp_convinteq} is not just an existence and uniqueness theorem for the solution of $\conv{a}$-convolution integral equations: under a mild assumption, \eqref{eq:confhyp_convinteq_sol_closform} provides an explicit expression for the solution which involves integration with respect to the parameters of the confluent hypergeometric function. However, if we are able to determine a closed-form expression for the function $\eta \in L^{a,\nu}$ which satisfies \eqref{eq:confhyp_convinteq_wiener}, then the representation \eqref{eq:confhyp_convinteq_sol} yields a more tractable explicit expression for the solution which does not involve index integrals. This is illustrated in the following corollary:

\begin{corollary} \label{cor:confhyp_convinteq2}
If $h \in L^{n + {1 \over 2},\nu}$ where $n \in \mathbb{N}_0$ and $0 \leq \nu < {1 \over 2}$, then the integral equation
\begin{equation} \label{eq:confhyp_convinteq_cor2}
f(x) + {n! \over \pi} \int_0^\infty \Bigl({x \over y}\Bigr)^{\!n+1} \Psi(n+1, 2; x+y) \, f(y) \,e^{-y} dy = h(x),
\end{equation}
has a unique solution $f \in L^{n+{1 \over 2},\nu}$, which is given by
\[
f(x) = h(x) + (h \conv{n+{1 \over 2}} \eta_n)(x) = h(x) + \! \int_0^\infty \! \int_0^\infty \! q_{n + {1 \over 2}}(x,y,\xi)\, h(y) \, \eta_n(\xi) \, \mathrm{m}_{n+{1 \over 2}}(y) dy \, \mathrm{m}_{n+{1 \over 2}}(\xi) d\xi
\]
where
\begin{equation} \label{eq:confhyp_convinteq_cor2_eta}
\eta_n(x) := \pi^{-{3 \over 2}\,} n! \, \Gamma(\tfrac{3}{2}+n) \, x^{{3 \over 2}+n} \sum_{k=0}^n {(-1)^{k+1} \over ({1 \over 2} + k) \, k! \, (n-k)!} \Psi(\tfrac{1}{2},1-k;x).
\end{equation}
\end{corollary}

We observe that \eqref{eq:confhyp_convinteq_cor2} is a natural generalization of the so-called Lebedev integral equation
\[
\bm{f}(x) + \sqrt{2 \over \pi^3} \int_0^\infty \bm{f}(y) {e^{-x-y} \over x+y} dy = \bm{h}(x),
\]
which is obtained in the case $n = 0$ (via the elementary transformation $\bm{f}(x) = x^{-1} e^{-x} f(2x)$,\, $\bm{h}(x) = x^{-1} e^{-x} h(2x)$). This Lebedev integral equation was introduced in \cite{lebedev1949}; it is a Kontorovich-Lebedev convolution equation, and its solution was derived in \cite[Section 17.1]{yakubovichluchko1994}. The existence of a closed-form solution for the generalized Lebedev equation \eqref{eq:confhyp_convinteq_cor2} is noteworthy because the function $\Psi(a,2;\cdot)$ (and the closely related Whittaker function $W_{\alpha,{1 \over 2}}$) is a particular case of the confluent hypergeometric function which is often encountered in problems in physics and chemistry \cite{laurenzi1973}.

\begin{proof}[Proof of Corollary \ref{cor:confhyp_convinteq2}]
The integral equation \eqref{eq:confhyp_convinteq_cor2} is the particular case of \eqref{eq:confhyp_convinteq_cor1} which is obtained by setting $a = n + {1 \over 2}$, $\beta = n+1$ and $\lambda = {n! \over \pi}$. In this case, $(\varPsi_{n + {1 \over 2}} \theta)(\tau) = {1 \over \pi} \Gamma({1 \over 2} + i \tau) \Gamma({1 \over 2} - i \tau) = {1 \over \cosh(\pi \tau)}$. Clearly, if $|\Im \tau| < {1 \over 2}$ then $\mathrm{Re}[\cosh(\pi\tau)] > 0$, hence the solvability condition $1 +  (\varPsi_{n + {1 \over 2}} \theta)(\tau) \neq 0$ holds in the strip $|\Im \tau| \leq \nu < {1 \over 2}$ and, according to Theorem \ref{thm:confhyp_convinteq}, the unique solution of \eqref{eq:confhyp_convinteq_cor2} is the function $f(x) = h(x) + (h \conv{ \raisebox{-2.5pt}{$\kern-.2em \scriptstyle n+\mathsmaller{1 \over 2} \kern-.28em$} } \eta)(x)$, where $\eta$ is the function satisfying 
\[
(\varPsi_{n + {1 \over 2}} \eta)(\tau) = {1 \over 1 + (\varPsi_{n + {1 \over 2}} \theta)(\tau)} - 1 = - {1 \over 2\cosh^2({\pi\tau \over 2})}.
\]
It remains to show that the function \eqref{eq:confhyp_convinteq_cor2_eta} satisfies this requirement. Using integral 2.16.48.14 of \cite{prudnikovII1986} and recalling the identity \eqref{eq:prel_Whit_modbessel}, we find that
\begin{equation} \label{eq:tilW_conv_integraleq_cor2_pf1}
\int_0^\infty \! {1 \over 2 \cosh^2({\pi \tau \over 2})}\, x^{-{1 \over 2} + i\tau} \Psi(\tfrac{1}{2}+i\tau, 1+2i\tau; x) \, \rho_{1/2}(\tau) d\tau = \pi^{-1} x^{1 \over 2} \Psi(\tfrac{1}{2}, 1; x).
\end{equation}
Now, by \eqref{eq:prel_confhyp_differentiate2} and the recurrence relation for the Gamma function we have
\begin{align*}
& \bigl| \Gamma\bigl(\tfrac{1}{2} + i\tau\bigr) \bigr|^2 {d^n \over dx^n} \Bigl[x^{-{1 \over 2}+n+i\tau} \Psi\bigl(\tfrac{1}{2}+i\tau,1+2i\tau;x\bigr)\Bigr] \\
& \quad = \bigl| \Gamma\bigl(\tfrac{1}{2} + i\tau\bigr) \bigr|^2 \bigl| \bigl(\tfrac{1}{2} + i\tau\bigr)_n \bigr|^2 x^{-\tfrac{1}{2} + i\tau} \Psi\bigl( \tfrac{1}{2}+n+i\tau, 1+2i\tau; x \bigr) \\
& \quad = \bigl| \Gamma\bigl(\tfrac{1}{2} + n + i\tau\bigr) \bigr|^2 x^{-\tfrac{1}{2} + i\tau} \Psi\bigl( \tfrac{1}{2}+n+i\tau, 1+2i\tau; x \bigr).
\end{align*}
Therefore, multiplying both sides of \eqref{eq:tilW_conv_integraleq_cor2_pf1} by $x^n$ and then applying ${d^n \over dx^n}$, we obtain
\begin{align*}
\int_0^\infty \! {1 \over 2\cosh^2({\pi \tau \over 2})} \, x^{-{1 \over 2} + i\tau} \Psi(\tfrac{1}{2}+n+i\tau, 1+2i\tau; x)\, \rho_{n+1/2}(\tau) d\tau = \pi^{-1}\, {d^n \over dx^n} \bigl[x^{{1 \over 2} + n} \Psi(\tfrac{1}{2},1;x)\bigr] = \quad \\
= \pi^{-{3 \over 2}} n! \, \Gamma(\tfrac{3}{2}+n) \, x^{1 \over 2} \sum_{k=0}^n {(-1)^k \over ({1 \over 2} + k) \, k! \, (n-k)!} \Psi(\tfrac{1}{2},1-k;x)
\end{align*}
where the last equality is obtained via Leibniz's rule, using the identities ${d^{n-k} \over dx^{n-k}} x^{1/2+n} = {\Gamma({3 \over 2}+n) \over \Gamma({3 \over 2}+k)} x^{{1 \over 2} + k}$ and \eqref{eq:prel_confhyp_differentiate1} (the possibility of differentiating under the integral sign being justified as in the proof of Lemma \ref{lem:confhyp_prodform_functeq}). If we now multiply both sides by $x^{1+n}$ and recall the notation \eqref{eq:confhyp_inverse}, we obtain $\bigl[ \varPsi_{n + {1 \over 2}}^{-1} \bigl(2 \cosh^2({\pi \tau \over 2})\bigr) \bigr](x) = - \eta_n(x)$. Consequently, $(\varPsi_{n + {1 \over 2}} \eta_n)(\tau) = - 2 \cosh^2({\pi \tau \over 2})$, as was to be proved. 
\end{proof}

\section*{Acknowledgments}

The first and third authors were partly supported by CMUP (UID/MAT/00144/2019), which is funded by Fundação para a Ciência e a Tecnologia (FCT) (Portugal) with national (MEC), European structural funds through the programmes FEDER under the partnership agreement PT2020, and Project STRIDE -- NORTE-01-0145-FEDER-000033, funded by ERDF -- NORTE 2020. The first author was also supported by the grant PD/BD/135281/2017, under the FCT PhD Programme UC|UP MATH PhD Program. The second author was partly supported by FCT/MEC through the project CEMAPRE -- UID/MULTI/00491/2013.


\end{document}